\documentclass[10pt,a4paper]{amsart}
\usepackage{tikz,tkz-graph}
\usetikzlibrary{arrows}
\usetikzlibrary{calc}
\usepackage{pgfplots}
\pgfplotsset{every axis/.append style={
                    axis x line=middle,    
                    axis y line=middle,    
                    axis line style={-,color=blue}, 
                    xlabel={$x$},          
                    ylabel={$y$},          
            }}
            
\usepackage[latin1]{inputenc}
\usepackage{amsmath}
\usepackage{amsthm}
\usepackage{amsfonts}
\usepackage{amssymb}
\usepackage{enumerate}
\usepackage{graphicx}
\usepackage{hyperref}
\usepackage{nicefrac}
\usepackage{blkarray}
\usepackage{cancel}
\usepackage{subfigure}

\DeclareMathOperator{\Sing}{Sing}

\DeclareMathOperator{\Mat}{Mat}
\DeclareMathOperator{\Spec}{Spec}

\DeclareMathOperator{\corank}{corank}
\DeclareMathOperator{\rank}{rank}
\DeclareMathOperator{\Deck}{Deck}
\DeclareMathOperator{\Ann}{Ann}

\DeclareMathOperator{\lk}{lk}
\DeclareMathOperator{\QP}{\textsc{qp}}
\def\A{\mathbb{A}}

\def\ZZ{\mathbb{Z}}
\def\TT{\mathbb{T}}

\def\CC{\mathbb{C}}

\def\FF{\mathbb{F}}

\def\PP{\mathbb{P}}

\def\cC{\mathcal{C}}

\def\G{\mathcal G_{\textsc{qp}}}

\def\uno{\bar{\mathbf{1}}}

\def\rightmap#1{\smash{\mathop{\rightarrow}\limits^{#1}}}

\def\injmap#1{\smash{\mathop{\hookrightarrow}\limits^{#1}}}
\def\surj{\mathbin{\rightarrow\mkern-6mu\rightarrow}}

\newcommand\enet[1]{\renewcommand\theenumi{#1}
\renewcommand\labelenumi{\theenumi}}

\newcommand\eneti[1]{\renewcommand\theenumii{#1}
\renewcommand\labelenumii{\theenumii}}

\newcommand\abrir{\makeatletter\renewcommand{\p@enumii}{}\makeatother}
\newcommand\cerrar{\makeatletter\renewcommand{\p@enumii}{\theenumi}\makeatother}
\def\es@roman#1{\protect{\romannumeral#1}}%

\newtheorem{thm}{Theorem}[section]  
\newtheorem{main-thm}{Theorem}  
\newtheorem{prop}[thm]{Proposition}
\newtheorem{main-conj}{Conjecture}%
\newtheorem{cor}[thm]{Corollary}
\newtheorem{lem}[thm]{Lemma}

\theoremstyle{remark}
\newtheorem{obs}[thm]{Remark}
\newtheorem{notation}[thm]{Notation}
\newtheorem{dfn}[thm]{Definition}
\newtheorem{exam}[thm]{Example}

\title{Quasi-projectivity of even Artin groups}

\author[R. Blasco]{Rub\'en Blasco-Garc{\'\i}a}
\address{Departamento de Matem\'aticas, IUMA\\ 
Universidad de Zaragoza\\ 
C.~Pedro Cerbuna 12\\ 
50009 Zaragoza, Spain} 
\email{rubenb@unizar.es}

\author[J.I. Cogolludo]{Jos{\'e} Ignacio Cogolludo-Agust{\'i}n}
\address{Departamento de Matem\'aticas, IUMA\\ 
Universidad de Zaragoza\\ 
C.~Pedro Cerbuna 12\\ 
50009 Zaragoza, Spain} 
\email{jicogo@unizar.es} 

\begin{document}

\thanks{The first author is partially supported by 
the Departamento de Industria e Innovación del Gobierno de Aragón and Fondo Social Europeo
PhD grant and the Spanish Government MTM2015-67781-P(MINECO/FEDER).
The second author is partially supported by 
the Spanish Government MTM2013-45710-C2-1-P and 
\emph{E15 Grupo Consolidado Geometr\'{\i}a} from the Gobierno de Aragón.} 

\subjclass[2010]{Primary: 20F36, 14F45; Secondary: 14H30, 57M05, 32S50} 

\begin{abstract}
Even Artin groups generalize right-angled Artin groups by allowing the labels in the defining 
graph to be even. In this paper a complete characterization of quasi-projective even Artin groups 
is given in terms of their defining graphs. Also, it is shown that quasi-projective even Artin 
groups are realizable by $K(\pi,1)$ quasi-projective spaces.
\end{abstract}

\maketitle

\section*{Introduction}
\label{sec-intro}
As suggested in \cite{catanese-fibred} a group is referred to as being quasi-projective 
(resp. quasi-K\"ahler) if it is the fundamental group of a smooth, connected, quasi-projective 
(resp. quasi-K\"ahler) space, that is, the complement of a hypersurface in a projective 
(resp. K\"ahler) variety.
The question of classification of quasi-projective groups, which today is referred to as Serre's
question, has been frequently alluded to since Zariski~\cite{Zariski-problem} and 
Van Kampen~\cite{VanKampen-fundamental} proposed it for complements of curves in the projective 
plane. The search for properties of such groups goes back to Enriques~\cite{Enriques-sulla} and
O.Zariski~\cite[Chapter VIII]{Zariski-algebraic}. This has developed in the search for obstructions 
for a group to be quasi-projective (resp. quasi-K\"ahler) starting with 
Morgan~\cite{Morgan-algebraic-topology}, Kapovich-Millson~\cite{Kapovich-Millson},
Arapura~\cite{Arapura-geometry,Arapura-fundamentalgroup}, Libgober~\cite{Libgober-groups}, 
Dimca~\cite{Di3}, 
Dimca-Papadima-Suciu~\cite{Dimca-Papadima-Suciu-jump-loci}, 
and Artal-Matei and the second author~\cite{ACM-artin,ACM-quasi-projective}.

In this paper we concentrate in the possible characterization of quasi-projective Artin groups, as 
stated in~\cite[p.~451]{Dimca-Papadima-Suciu-jump-loci}.
Any proof of such results requires the use of obstructions to disregard the negative cases as 
well as the constructive part of finding realizations for the positive cases.

A first approach to this problem is given in~\cite[Thm. 11.7]{Dimca-Papadima-Suciu-jump-loci} 
where quasi-projective right-angled Artin groups are characterized by complete multipartite 
graphs corresponding to direct products of free groups.
In the more general case of even Artin groups, that is, Artin groups associated with 
even-labeled graphs, the label plays an important role
and not all multipartite graphs produce quasi-projective Artin groups.

In order to describe such graphs we define the concept of \emph{$\QP$-irreducible graph}. 
In this context, \emph{graph} means for us simple graph.
Let us denote by $\G$ the family of labeled graphs whose associated Artin groups are 
quasi-projective. Given two labeled graphs $\Gamma_1=(V_1,E_1,m_1)$, 
$\Gamma_2=(V_2,E_2,m_2)$, we define their 2-join $\Gamma_1*_2 \Gamma_2=(V,E,m)$ as the 
labeled graph given by the join of $\Gamma_1$ and $\Gamma_2$ whose connecting 
edges have all label~2, that is
$$
m(e)=
\begin{cases}
m_i(e) & \text{ if } e\in E_i\\
2 & \text{ if } e\in E\setminus (E_1\cup E_2).
\end{cases}
$$
We say $\Gamma\in \G$ is a $\QP$-irreducible graph if $\Gamma$ is not a 2-join of two graphs in~$\G$.

Denote by $\overline{K}_r$ a disjoint graph with $r$ vertices and no edges. 
Also denote by $S_{m}$ the graph given by two vertices joined by an edge with label $m$. 
Finally, denote by $T(4,4,2)$ the triangle as shown in Figure~\ref{fig-simple-blocks}. 
It will be shown that these are the only $\QP$-irreducible even graphs. 
In other words, the main result of this paper is the following.

\begin{main-thm}\label{thm-qp}
Let $\Gamma=(V,E,2\ell)$ be an even-labeled graph and $\A_\Gamma$ its associated even Artin group. 
Then the following are equivalent:
\begin{enumerate}
 \item $\A_\Gamma$ is quasi-projective, that is, $\Gamma\in \G$.
 \item $\Gamma$ is the 2-join of finitely many copies of $\overline{K}_r$, $S_{2\ell}$, and $T$.
\end{enumerate}
Moreover, if $\Gamma\in \G$, then $\A_\Gamma=\pi_1(X)$ where $X=\PP^2\setminus \cC$ is a 
curve complement.
\end{main-thm}

\begin{center}
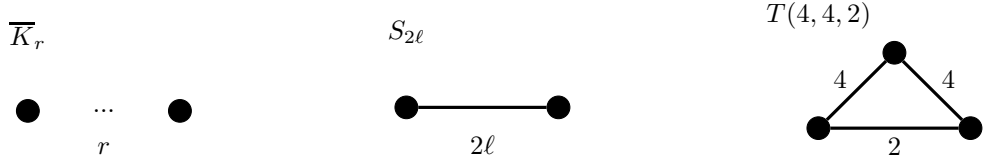
\begin{figure}
\begin{tabular}{ccccc}
\begin{tikzpicture}[vertice/.style={draw,circle,fill,minimum size=0.3cm,inner sep=0}]
\node[vertice] at (1,1) {};
\node[vertice] at (-1,1) {};
\node at (0,1) {$...$};
\node at (0,.5) {$r$};
\node at (-1,2) {$\overline{K}_r$};
\end{tikzpicture}
& \quad \quad  \quad \quad \quad \quad&
\begin{tikzpicture}[vertice/.style={draw,circle,fill,minimum size=0.3cm,inner sep=0}]
\node[vertice] (D1) at (1,1) {};
\node[vertice] (F1) at (-1,1) {};
\draw[very thick] (D1)--(F1);
\node at (0,.5) {$2\ell$};
\node at (-1,2) {$S_{2\ell}$};
\end{tikzpicture}
& \quad \quad  \quad \quad \quad \quad&
\begin{tikzpicture}[vertice/.style={draw,circle,fill,minimum size=0.3cm,inner sep=0}]
\node[vertice] (D1) at (1,.5) {};
\node[vertice] (E1) at (0,1.5) {};
\node[vertice] (F1) at (-1,.5) {};
\draw[very thick] (D1)--(F1)--(E1)--(D1);
\node[left] at (-0.5,1.15) {$4$};
\node[right] at (0.5,1.15) {$4$};
\node[below] at (0,.5) {$2$};
\node at (-1,2) {$T(4,4,2)$};
\end{tikzpicture}
\end{tabular}
\caption{$\QP$-irreducible graphs of type $\overline{K}_r$, $S_{2\ell}$, and $T(4,4,2)$.}
\label{fig-simple-blocks}
\end{figure}
\end{center}

The $K(\pi,1)$ conjecture referred to an Artin group $\A_\Gamma$ claims that a certain space, that 
appears as a quotient of the complement of the Coxeter arrangement by the action of the Coxeter group 
associated to $\Gamma$ is an Eilenberg-MacLane space whose fundamental group is $\A_\Gamma$ --\,or a 
$K(\A_\Gamma,1)$ space --\,see for instance~\cite{Paris-lectures} for a detailed explanation of this 
conjecture. In the context of quasi-projective groups, we can also ask ourselves 
whether or not a quasi-projective Artin group is realizable by an Eilenberg-MacLane space.

\begin{main-conj}[Quasi-projective $K(\pi,1)$ conjecture]
Any quasi-projective Artin group $\A_\Gamma$ can be realized as $\A_\Gamma=\pi_1(X)$ for a smooth, 
connected, quasi-projective Eilenberg-MacLane space~$X$.
\end{main-conj}

The other main result of this paper is a positive answer to the quasi-projective $K(\pi,1)$ 
conjecture for even Artin groups.

\begin{main-thm}
\label{thm-kpi1}
Quasi-projective even Artin groups satisfy the quasi-projective $K(\pi,1)$ conjecture.
\end{main-thm}

This paper is organized as follows: in section~\ref{sec-settings} the general definitions of 
(even) Artin groups and quasi-projective groups will be given as well as the notion of 
characteristic varieties as an invariant of a group. Section~\ref{sec-preliminaries} will 
be devoted to studying kernels of cyclic quotients of Artin groups, called co-cyclic subgroups. 
Section~\ref{sec-irreduicible} focuses on the problem of finding $\QP$-irreducible graphs. 
The main theorems will be proved in section~\ref{sec-proofs}.

\section{Settings and definitions}
\label{sec-settings}
\subsection{Artin groups}
Artin groups are an interesting family of groups from both an algebraic and a topological 
point of view.

We recall the definition of the Artin group associated with a labeled graph $\Gamma=(V,E,\ell)$, 
where $(V,E)$ is a graph and $m:E\to \mathbb Z_{\geq 2}$ is the label map 
assigning an integer $m_e=m(e)\in \mathbb Z_{\geq 2}$ to each edge $e\in E$ of~$\Gamma$.

\begin{dfn}(Artin groups)
Let $\Gamma=(V,E,m)$ be a labeled graph, the \emph{Artin group} $\A_{\Gamma}$
associated with $\Gamma$ has the following presentation:
\begin{equation}
\label{eq-Artingroup}
\A_{\Gamma}=\langle v; v\in V\mid \langle uv\rangle^{m_e}=
\langle vu\rangle^{m_e}, e=\{u,v\}\in E \rangle.
\end{equation}
where $\langle uv\rangle^{m_e}$ is the alternating product of length $m_e$ 
beginning with $u$.
$$\langle uv\rangle^{m_e}=\underbrace{(uv...)}_{m_e}.$$
Note that $\langle u,v\rangle^{2\ell}=(uv)^{\ell}$.

Right-angled Artin groups are defined as Artin groups in which all the edges of the graph 
have label~2.
\end{dfn}

\begin{obs}
As a word of caution for the reader, note that non-adjacent vertices have no 
associated relation. This notation differs in other contexts where non-adjacencies
are replaced by $\infty$-labeled edges, edges with label 2 are removed, and labels 
3 are erased.
\end{obs}

One special subfamily of Artin groups which we will use along this text is the family of 
even Artin groups.

\begin{dfn}(Even Artin groups)	
We say that an Artin group associated with $\Gamma=(V,E,m)$ is \emph{even} 
if its labels $m_e$, $e\in E$ are all even numbers. This will be oftentimes
denoted as~$\Gamma=(V,E,2\ell)$.
\end{dfn}

\subsection{Quasi-projective groups}
The main focus of this paper is the study of those groups that can appear as 
fundamental groups in an algebraic geometry context, in particular as fundamental
groups of smooth connected quasi-projective varieties. 
Recall that a \emph{quasi-projective} variety is the complement of a hypersurface in a 
projective variety defined simply as the zero locus of a finite number of homogeneous 
polynomials in~$\CC[x_0,...,x_n]$.

\begin{dfn}
A group $G$ is \emph{quasi-projective} if $G=\pi_1(X)$ for a smooth connected 
quasi-projective variety~$X$.
\end{dfn}

\begin{exam}
\label{exam-abelian}
Since the fundamental group of the complement of a smooth plane curve of degree $d$ in 
$\PP^2$ is the cyclic group $\ZZ_d$ and the complement of two lines in $\PP^2$ has the
homotopy type of $\CC^*$, all cyclic groups are quasi-projective. Moreover, since the 
complement of $r+1$ irreducible smooth curves $C_0,...,C_r$ of degrees $d_i=\deg C_i$ 
intersecting transversally has fundamental group
$$\pi_1(\PP^2\setminus C_0\cup...\cup C_r)=\ZZ^r\oplus \ZZ_d,$$
where $d=\gcd(d_0,...,d_r)$, one immediately obtains that all abelian groups are quasi-projective.
\end{exam}

This example points out that the quasi-projective variety whose fundamental group
realizes a quasi-projective group is clearly not unique in any geometrical sense, since the
torsion part $d$ can be attained in many different ways.

\begin{exam}
\label{exam-free}
On the other end of abelianization properties, the free group of rank $r$ is also
quasi-projective since it can be realized as the fundamental group of the complement of $r+1$ 
points in the complex projective line~$\PP^1$.
\end{exam}

The following important properties of quasi-projective groups are well known.

\begin{prop}
\label{prop:QP}
\mbox{}
\begin{enumerate}
\item \label{QPGL}
If $G$ is a quasi-projective group and $K\subset G$ is a finite index subgroup of $G$,
then $K$ is also a quasi-projective group.
\item \label{QP2J}
If $G_1,G_2$ are quasi-projective groups, then $G_1\times G_2$ is also a quasi-projective group.
\end{enumerate}
\end{prop}

%
%
%

\subsection{Serre's question for Artin groups}
The question about deciding whether a certain group is quasi-projective is known as 
\emph{Serre's question}. This question is solved for right-angled Artin groups, 
but almost nothing is known for more general Artin groups.

\begin{thm}[{\cite[Theorem 11.7]{Dimca-Papadima-Suciu-jump-loci}}]
\label{Suciu}
The right-angled Artin group $\A_\Gamma$ is quasi-projective 
if and only if $\A_\Gamma$ is a product of finitely generated free groups, 
i.e. $\A_\Gamma =\FF_{n_{1}} \times ...  \times \FF_{n_{r}}$.
\end{thm}

The direct implication is proved by exploiting the obstructions of resonance varieties of
quasi-projective groups. The converse is achieved by realizing such groups as fundamental 
groups of quasi-projective varieties, built as products of complements of points in~$\CC$.

In fact, this result can be interpreted in terms of the graphs via the 2-join construction
as follows.

\begin{dfn}\label{def-2join}
Consider $\Gamma_1$ and $\Gamma_2$ two labeled graphs. 
The \emph{2-join} of $\Gamma_1$ and $\Gamma_2$, denoted by $\Gamma=\Gamma_1*_2 \Gamma_2$ 
is the labeled graph $\Gamma$ defined as the join of the graphs and whose label 
is defined as
$$
m(e)=\begin{cases} 
m_i(e) & \text{ if } e\in E(\Gamma_i)\\
2 & \text{ otherwise.}
\end{cases}
$$
\end{dfn}
The Artin group of a 2-join is the product of the Artin groups, that is,
\begin{equation}
\label{eq-2join}
\A_{\Gamma_1*_2\Gamma_2}=\A_{\Gamma_1}\times \A_{\Gamma_2}.
\end{equation}
From Example~\ref{exam-abelian} note that a free abelian group of rank $r$ is an Artin group
corresponding to a complete right-angled Artin group of $r$ vertices, or 2-joins of $r$ points.
From Example~\ref{exam-free} note that a free group of rank $r$ is also an Artin group
corresponding to a totally disconnected graph of $r$ vertices. Using~\eqref{eq-2join},
Theorem~\ref{Suciu} can be rewritten as follows.

\begin{thm}
\label{Suciu2}
Let $\Gamma$ be a right-angled graph, then $\A_{\Gamma}$ is quasi-projective 
if and only if $\Gamma$ is the $2$-join of finitely many totally disconnected graphs.
\end{thm}

For triangle Artin groups and general type Artin groups, partial results on their 
quasi-projectivity are given in~\cite{ACM-artin}, among those we describe the following 
associated with Artin groups of type $\A_{S_{2\ell}}$ and $\A_T$ in Figure~\ref{fig-simple-blocks}.

\begin{thm}[{\cite[Chapter 5]{ACM-artin}}]
\label{BP}
The Artin groups $\A_{S_{2\ell}}=\langle a,b\mid (ab)^\ell=(ba)^\ell\rangle$ and
$\A_T=\langle a,b,c\mid abab=baba, acac=caca, bc=cb\rangle$
are quasi-projective.
\end{thm}

Our objective in this paper is to give a similar characterization 
to Theorem~\ref{Suciu2} for even Artin groups.

\subsection{Characteristic Varieties}
\label{sec:charvar}
Characteristic varieties are a sequence of invariants of a group. They were 
introduced by Hillman in~\cite{Hillman-alexander} for links using Alexander modules and 
systematically studied by Cohen-Suciu~\cite{Cohen-Suciu-characteristic} for hyperplane arrangement 
complements, Libgober~\cite{Libgober-characteristic} for plane curve complements and from a different 
point of view by Arapura in~\cite{Arapura-geometry} for K\"ahler manifolds using jumping loci 
of cohomology of local systems. It should also mentioned that the connection between Alexander 
modules and cohomology of local systems was first proved by E.Hironaka~\cite{Hironaka-alexander}.

For expository reasons we will mainly follow~\cite{Libgober-characteristic} and we will only provide 
specific references for the more specialized results.
Let $X$ be a finite CW-complex and $G=\pi_1(X)$ its fundamental 
group. For the sake of simplicity, we assume that the abelianization $H_{1}(G)=G/G'$ of $G$ 
is torsion-free, say~$H_{1}(G)=\ZZ^r$. Consider the universal abelian cover 
$\tilde X\rightmap{\phi} X$, where $\Deck(\phi)=\mathbb{Z}^{r}$ is generated by 
$t_{1},\dots,t_r\in \Deck(\phi)$. 
Since $\Deck(\phi)$ acts on $H_1(\tilde X)$, the group $H_1(\tilde X)$ inherits a 
module structure over the ring $\Lambda=\ZZ[\Deck(\phi)]=\ZZ[\mathbb{Z}^{r}]$.
This module $M_X=H_1(\tilde X)$ is called the Alexander module of $X$. As any $\Lambda$-module, 
$M_X$ has a sequence of invariants given by the Fitting ideals or analogously by the sequence 
of annihilators of its exterior powers as follows:
$$I_k=\Ann_{\Lambda}\left(\bigwedge^{k} M_X\right)\subset \Lambda,$$
where $\Ann_R(A)=\{r\in R \mid ra=0\ \forall\, a\in A\}\subset R$ is by definition 
the annihilator ideal of an $R$-module~$A$.
After tensoring $\Lambda$ by $\CC$, a new ring $\Lambda^\CC$ is obtained over which 
one can take an algebraic geometrical point of view and consider the zero locus of 
$I_r\otimes \Lambda^\CC$ inside the torus $\Spec \Lambda^\CC=(\CC^*)^r$.

\begin{dfn}
We define the sequence of \emph{characteristic varieties} of $X$ as:
$$V_{1}(X):=Z(I_1)\supset ... \supset V_{k}(X):=Z(I_k) \supset ...$$
where $Z(I_k)\subset (\CC^*)^r$ is the zero locus of~$I_k$.
\end{dfn}

There is an alternative way to define characteristic varieties using Fitting ideals.

\begin{dfn}
Let $\varphi: A_2\rightarrow A_1$ be a map of free modules over a ring~$R$. 
We define the ideal $\tilde F_k(\varphi)\subset R$ as the image of the canonical map:
$$\bigwedge^k A_2 \otimes \bigwedge^k A_1^* \rightarrow R$$
induced by~$\varphi$.
\end{dfn}

\begin{dfn}
Let $M$ be a finitely presented module over $R$ and consider a free resolution 
$$\varphi: A_2\rightarrow A_1\rightarrow M \rightarrow 0$$
of $M$ such that $A_1$ (resp. $A_2$) is a finitely generated $R$-module of rank $r$
(resp. $s$). 
For every integer $k\geq 0$ we define the $k$-th Fitting ideal of $M$ to be
$$F_k(M):=\tilde F_{r-k} (\varphi).$$
\end{dfn}

\begin{prop}
Under the above conditions, the sequence of characteristic varieties $V_k(X)$ 
coincides with the zero locus of the Fitting ideals of its Alexander module~$F_k(M_X)$.
\end{prop}

\begin{proof}
This is an immediate consequence of~\cite[Cor. 1.3]{buchsbaum-eisenbud-annihilates}.
\end{proof}

Characteristic varieties of quasi-projective spaces satisfy the following result:

\begin{prop}{\cite{Arapura-geometry,dimca-papadima-suciu-essential}}
\label{Fin}
The irreducible components of the characteristic varieties associated to a 
quasi-projective group $G$ are algebraic translated tori by torsion points, that is, 
they are intersection of zero-sets of polynomials of the form
$$P(t_1,...,t_r)=\prod_i (t_1^{n_1}...t_r^{n_r}-\nu_i),$$
where $\nu_i$ is a root of unity.

Moreover, the intersection of two such irreducible components 
is a finite union of torsion points.
\end{prop}

From the computational point of view, a third way to calculate the sequence of 
characteristic varieties from a finite presentation of a group 
\begin{equation}
\label{eq:presG}
G=\pi_1(X)=\langle a_1,...,a_n : R_1=...=R_m=1\rangle
\end{equation}
is provided via Fox calculus (cf.~\cite{Fox-free1}). 

Formally, one associates a matrix 
$$A=\left( \frac{\partial R_i}{\partial a_j}\right)_{1\leq i\leq m, 1\leq j\leq n},$$
to the presentation~\eqref{eq:presG}, where the derivative of a word in the letters 
$a_1,\dots,a_n$ is obtained by extending the following defining properties by linearity:
$$
\frac{\partial uv}{\partial a_j}=
\frac{\partial u}{\partial a_j}+\phi(u)\frac{\partial v}{\partial a_j}, \quad\quad
\frac{\partial 1}{\partial a_j}=0,\quad \quad \text{and} \quad\quad
\frac{\partial a_i}{\partial a_j}=
\begin{cases}
1 & \text{ if } i=j\\
0 & \text{ otherwise.}
\end{cases}
$$

The matrix $A$ is called the \emph{Alexander matrix} associated with~\eqref{eq:presG}
and it turns out to be the matrix of the free resolution of a module which is not 
the Alexander module, but the \emph{Alexander invariant} 
$\tilde M_X=H_1(\tilde X,\phi^{-1}(p))$, which is the relative homology of the universal 
abelian cover of $X$ relative to the preimage of a point as a $\Lambda$-module exactly
as was done for the Alexander module $M_X$.
As in knot theory, both invariants are related (see for instance~\cite[Ch. 1]{ji-tesis}).

\begin{prop}
The sequence of characteristic varieties of $X$ can be calculated via Fox calculus as
$$V_k(X)\setminus \uno=Z(F_{k+1}(M_X))\setminus \uno=Z(F_{k+1}(\tilde M_X))\setminus \uno.$$
\end{prop}

The computational advantage of $F_{k+1}(\tilde M_X)$ is that it can be computed from 
the Alexander matrix $A$ of a free resolution of $\tilde M_X$ as follows:
$$
F_{k+1}(\tilde M_X)=
\begin{cases}
\Lambda & \text{ if } k>n\\
0 & \text{ if } k \leq \max\{0,n-m\}  \\
(\text{minors of order } n-k \text{ of } A)  & \text{ otherwise. } \\
\end{cases}
$$

\section{Preliminaries}
\label{sec-preliminaries}
Characteristic varieties of even Artin groups are too similar to those of quasi-projective
groups and hence they cannot be used to tell them apart. However, some of their finite 
index subgroups can be detected as not quasi-projective. This is why we present a study 
of a certain type of subgroups of even Artin groups that will be key in the discussion
on quasi-projectivity.

\subsection{Co-cyclic subgroups of even Artin groups}
\label{Subgroups}
Let us consider the even Artin group associated with $\Gamma=(V,E,2\ell)$.

The Artin group associated with $\Gamma$ has a presentation 
$\A_{\Gamma}=\langle v; v\in V\mid A_{\ell_e}(e); e\in E\rangle$, 
where $A_{\ell_e}(e)$ denotes the relation 
$(uv)^{\ell_e}=(vu)^{\ell_e}$ with $e=\{u,v\}$.
Let us fix a vertex, say $u\in V$ and an integer $k>1$; our purpose is to give a 
presentation of the index $k$ subgroup $\A_{\Gamma,u,k}$ of $\A_{\Gamma}$ 
defined as the kernel of the following morphism:

$$
\array{rcl}
\alpha_{u,k}: \A_{\Gamma} &\longrightarrow & \mathbb{Z}_{k} \\
v & \mapsto & \begin{cases} 1 \hspace{0.5 cm} \text{if } v=u\\
0 \hspace{0.5 cm} \text{otherwise.}\end{cases}
\endarray
$$
As suggested by an anonymous referee, one can think of these as finite index normal subgroups of a group 
that appear as the kernel of a surjection onto a finite cyclic group, and refer to them as \emph{co-cyclic subgroups}.

Note that, for any $v\in V$ the conjugation of $v$ by $u^i$ is in $\A_{\Gamma,u,k}$,
$v_i:=u^ivu^{-i}$. Also, $\bar{u}:=u^{k}$ will be in the kernel of $\alpha_{u,k}$. 
In order to write a presentation for $\A_{\Gamma,u,k}$ we need some notation. Let us denote by 
$\langle x,y\rangle_{i,\varepsilon}^l$ a formal word in the letters $\{x_0,...,x_{k-1},y\}$ as follows
$$\langle x,y\rangle_{i,\varepsilon}^l=
(x_{i}\cdots x_{k-1}yx_0\cdots x_{i-1})^cx_{i}\cdots x_{i+r-1}x_{i+r}^\varepsilon,$$
where $l=ck+r$, $i\in \ZZ_k$ and $\varepsilon=0,1$.
Note that $\langle x,y\rangle_{i,\varepsilon}^l$ can be thought of as a cyclic product 
of the letters $x_0,\dots,x_{k-1}$, and $y$ starting at $x_i$ and with length $c(k+1)+r+\varepsilon$.

Also, let us consider the set of vertices in $V$ adjacent to $u$ with label~$2$:
$$V_{2,u}=\{v\in V\mid e=\{u,v\}\in E, m_e=2\ell_e=2\}.$$
The remaining vertices will be denoted by $W=V\setminus (\{u\}\cup V_{2,u})$.

One obtains the following presentation for~$\A_{\Gamma,u,k}$.

\begin{thm}
\label{thm-pres}
The co-cyclic subgroup $\A_{\Gamma,u,k}$ is generated by 
$$\{\bar u\}\cup V_{2,u} \cup \bigcup_{w\in W}\{w_{0},...,w_{k-1}\}$$
and the following is a complete set of relations:
\begin{itemize}
\item[($R$)] 
\begin{enumerate}
\item\label{rel-uv} 
$A_1(v,\bar u)$, for $v\in V_{2,u}$,
\item\label{rel-vv} 
$A_{\ell_e}(v,v')$, for $v,v'\in V_{2,u}$, $e=\{v,v'\}\in E$,
\item\label{rel-vw} 
$A_{\ell_e}(v,w_i)$, for $v\in V_{2,u}$, $w\in W$, $i\in \ZZ_k$, $e=\{v,w\}\in E$,
\item\label{rel-ww} 
$A_{\ell_e}(w_i,w'_i)$, for $w,w'\in W$, $i\in \ZZ_k$, $e=\{w,w'\}\in E$.
\end{enumerate}
\item[($RB$)] $B^i_{\ell_e,k}(w,\bar u)$, for $w\in W\cap \lk(u)$, $i\in \ZZ_k$, $e=\{u,w\}\in E$,
\end{itemize}
where $\ell_e=c_ek+r_e$ and $B_{\ell_e}^{i}(w,\bar u)$ is the relation
$$
\langle w,\bar{u}\rangle_{i,\varepsilon}^{\ell_e}=
\langle w,\bar{u}\rangle_{i+1,\varepsilon}^{\ell_e}
\quad \quad \text{and} \quad 
\varepsilon=
\begin{cases}
0 & \text{if } 0\leq i<k-r_{e}\\
1 & \text{otherwise}.
\end{cases}
$$
\end{thm}

\begin{proof}
The proof is a direct application of Reidemeister-Schreier's theorem (c.f.~\cite{Fox-free3}) 
to obtain a presentation of $\A_{\Gamma,u,k}$ as the kernel of $\alpha_{u,k}$
$$
\A_{\Gamma,u,k}\ \injmap{j}\ \A_{\Gamma}\ \rightmap{\alpha_{u,k}}\ \ZZ_k.
$$
Consider the Reidemeister's section $s:\ZZ_k\to \A_{\Gamma}$ of the map $\alpha_{u,k}$ given as
$s(i):=u^i$. Then $\A_{\Gamma,u,k}$ admits a presentation generated by the letters
$$\{\bar u\}\cup \bigcup_{v\in V}\{v_{0},...,v_{k-1}\},$$
where $j(\bar u)=u^k$ and $j(v_i)=u^ivu^{-i}$ whose relations are:
\begin{enumerate}
\item $A_{\ell_e}(v_i,w_i)$, for $v,w\in V$, $i\in \ZZ_k$, if $e=\{v,w\}\in E$,
\item $B^i_{\ell_e,k}(w,\bar u)$, for $i\in \ZZ_k$, if $v\in \lk(u)$.
\end{enumerate}
However, note that if $v\in V_{2,u}$, then $u^ivu^{-i}=v_i=v_j=u^jvu^{-j}$ which implies 
a reduction in the set of generators, which now becomes as stated:
$$\{\bar u\}\cup V_{2,u} \cup \bigcup_{w\in W}\{w_{0},...,w_{k-1}\}.$$
Finally, note that the only relations affected by this elimination of generators are
those of type $A_{\ell_e}(v_i,w_i)$ for $v\in V_{2,u}$, which now become $A_{\ell_e}(v,w_i)$, 
and those of type $B^i_{1,k}(v,\bar u)$ for $v\in V_{2,u}$, which now be reduced to 
$A_{1}(v,\bar u)$, as stated.
\end{proof}

\begin{obs}
\label{obs-l2}
Our purpose will be to study the characteristic varieties of the co-cyclic subgroups.
As presented in section~\ref{sec:charvar} these are subvarieties of $\Spec \CC[G/G']$,
for $G=\A_{\Gamma,u,k}$. First we will describe the abelianization of $\A_{\Gamma,u,k}$. 
Since $G$ is finitely presented consider $\FF\to G$ the map from the free group $\FF$ 
in the generators of $G$. The kernel $K$ of this homomorphism is a free subgroup generated
by the set of relations in $G$. Consider $G\ \rightmap{\Phi_G}\ G/G'$, $g\to t_g$ the 
abelianization map (with a multiplicative structure). According to Theorem~\ref{thm-pres} 
the abelianization $G/G'=\Phi_\FF(\FF)/\Phi_\FF(K)$ is generated by 
$$\{t_{\bar u}\}\cup \{t_v\}_{v\in V_{2,u}} \cup \bigcup_{w\in W}\{t_{w,{0}},...,t_{w,{k-1}}\},$$
where for convenience $t_{w,i}$ is used to denote~$t_{w_i}$.
Note that ($R$).\eqref{rel-uv}-\eqref{rel-ww} considered as words in the free group $\FF$ 
belong in fact to $\FF'$ and hence their image by the abelianitation map $\Phi_\FF$ is trivial.
On the other hand, the words $B_{\ell_e,k}^{i}(w,\bar{u})$, $w\in W\cap \lk(u)$ produce the 
following relations in homology:
\begin{equation}
\label{eq-homology}
\array{ll}
t_{w,{i}}=t_{w,{i+d_e}}=...=t_{w,{i+nd_e}} & \text{ if } e=\{u,w\}, d_e=\gcd(\ell_e,k),
\endarray
\end{equation}
\end{obs}


\begin{dfn}
The presentation described in Theorem~\ref{thm-pres} will be referred to as the 
\emph{standard presentation} of~$\A_{\Gamma,u,k}$.
\end{dfn}

\subsection{Fox calculus on the co-cyclic subgroups~$\A_{\Gamma,u,k}$}

\subsubsection{Fox derivatives of a standard presentation}
We want to describe the Fox derivatives of the relations of a standard presentation 
of the subgroup~$\A_{\Gamma,u,k}$.

The first set of relations of type $(R)$ in Theorem~\ref{thm-pres} are classical 
Artin relations. In order to describe their Fox derivatives we introduce the polynomial
$p_{l}(t)=\frac{t^l-1}{t-1}$ and as above, we denote by $t_g$ the homology class of 
an element~$g$. In the following results we present the Fox derivatives of certain 
relations of type $W_1=W_2$, by this we mean the derivative of the abstract 
word~$W_1W_2^{-1}$.

\begin{lem}\label{lemma-foxA}
Under the above conditions
$$\frac{\partial A_{\ell_e}(a,b)}{\partial g} = 
\begin{cases}
-(t_{b}-1)p_{\ell_e}(t_{a}t_{b}) & \text{ if } g=a\\
(t_{a}-1)p_{\ell_e}(t_{a}t_{b}) & \text{ if } g=b\\
0 & \text{ otherwise. }
\end{cases}
$$
\end{lem}


In order to describe the derivatives of relations of type $(RB)$, let us use some conventions:


\begin{align*}
& \bar{t}_{w,i,j}=
\begin{cases} 
t_{w,i}\cdots t_{w,j-1} & \text{ if } 0\leq i<j\leq k\\ 
1 & \text{ if } i=j \\
\bar t_{w}/\bar t_{w,j,i} & \text{ if } 0\leq j<i\leq k\\ 
\end{cases}\\
\end{align*}
where $t_{w,i}=t_{w_i}$ with $w_i=u^i wu^{-i}$, $t_{0}=t_{\bar u}$, and $\bar{t}_{w}=\bar{t}_{w,0,k}$.

\begin{lem}
\label{lemma-foxB}
Under the above conditions
$$\dfrac{\partial B_{\ell_e,k}^{i}(w,\bar{u})}{\partial g}=
\begin{cases}
\bar{t}_{w,i,k}(1-t_{w,i+r_{e}+\varepsilon}^{-1})p_{c_{e}+\varepsilon}(t_0\bar{t}_{w}) 
& \text{ if } g=\bar u\\
\bar{t}_{w,i,j}(1-t_{w,i+r_e+\varepsilon}^{-1})p_{c_e}(t_0\bar{t}_{w}) & \text{ if } g=w_j, j<i \\
\left(1-\dfrac{t_0\bar{t}_{w}}{t_{w,i+r_e+\varepsilon}}\right)
p_{c_e}(t_0\bar{t}_{w})+(t_0\bar{t}_{w})^{c_e} & \text{ if } g=w_i \\
\bar{t}_{w,i,j}(1-t_{w,i+r_e+\varepsilon}^{-1})p_{c_e+1}(t_0\bar{t}_{w}) & 
\text{ if } g=w_j, i< j < i+r_e+\varepsilon \\ 
\bar{t}_{w,i,j}(1-t_{w,i+r_e+\varepsilon}^{-1})p_{c_e}(t_0\bar{t}_{w}) - \dfrac{\bar{t}_{w,i,i+r_e+\varepsilon}}{t_{w,i+r_e+\varepsilon}}(t_0\bar{t}_{w})^{c_e}& 
\text{ if } g=w_j, i< j = i+r_e+\varepsilon \\ 
\bar{t}_{w,i,j}(1-t_{w,i+r_e+\varepsilon}^{-1})p_{c_e}(t_0\bar{t}_{w}) & 
\text{ if } g=w_j, j< i+r_e+\varepsilon .
\end{cases}
$$
\end{lem}

\begin{proof}
The proof is straightforward. We will work out a sample case. Assume $g=w_i$, then
$$\dfrac{\partial B_{\ell_e,k}^{i}(w,\bar{u})}{\partial w_i}=\dfrac{\langle w,\bar{u}\rangle_{i,\varepsilon}^{\ell_e}
	(\langle w,\bar{u}\rangle_{i+1,\varepsilon}^{\ell_e})^{-1}}{\partial w_i}.$$

First let us calculate $\dfrac{\langle w,\bar{u}\rangle_{i,\varepsilon}^{\ell_e}}{\partial w_i}$. It is straight-forward that
$$\dfrac{\langle w,\bar{u}\rangle_{i,\varepsilon}^{\ell_e}}{\partial w_i}=p_{c_e}(t_0\bar{t}_w)+(t_0\bar{t}_w)^{c_e}$$
 
 $$\dfrac{\langle w,\bar{u}\rangle_{i+1,\varepsilon}^{\ell_e}}{\partial w_i}=\bar{t}_{w,i+1,i}p_{c_e}(t_0\bar{t}_w)$$
 
 Now, using the multiplication rule and
 $0=\frac{\partial uu^{-1}}{\partial v}=
 \frac{\partial u}{\partial v}+t_u\frac{\partial u^{-1}}{\partial v}$
 one obtains:
 
 $$\dfrac{(\langle w,\bar{u}\rangle_{i+1,\varepsilon}^{\ell_e})^{-1}}{\partial w_i}=-\dfrac{\bar{t}_{w,i+1,i}p_{c_e}(t_0\bar{t}_w)}{(t_0\bar{t}_w)^{c_e}\bar{t}_{w,i,i+r_e+\varepsilon}}.$$
 	
Therefore:

$$\dfrac{\partial B_{\ell_e,k}^{i}(w,\bar{u})}{\partial w_i}=p_{c_e}(t_0\bar{t}_w)+(t_0\bar{t}_w)^{c_e}+(t_0\bar{t}_w)^{c_e}\bar{t}_{w,i,i+r_e+\varepsilon}\dfrac{(\langle w,\bar{u}\rangle_{i+1,\varepsilon}^{\ell_e})^{-1}}{\partial w_i}=$$
$$=p_{c_e}(t_0\bar{t}_w)+(t_0\bar{t}_w)^{c_e}-(t_0\bar{t}_w)^{c_e}\bar{t}_{w,i,i+r_e+\varepsilon}\dfrac{(\bar{t}_{w,i+1,i})p_{c_e}(t_0\bar{t}_w)}{(t_0\bar{t}_w)^{c_e}\bar{t}_{w,i,i+r_e+\varepsilon}}=$$
$$=p_{c_e}(t_0\bar{t}_w)+(t_0\bar{t}_w)^{c_e}-\dfrac{(t_0\bar{t}_w)^{c_e}}{\bar{t}_{w,i,i+r_e+\varepsilon}}p_{c_e}=\left(1-\dfrac{t_0\bar{t}_{w}}{t_{w,i+r_e+\varepsilon}}\right)
p_{c_e}(t_0\bar{t}_{w})+(t_0\bar{t}_{w})^{c_e}.$$
\end{proof}

\subsubsection{Alexander matrices for co-cyclic subgroups of even Artin groups}
Given $\Gamma=(V,E,2\ell)$ an even labeled graph. Let us fix $u\in V$ and an integer $k>1$.
We will denote by $M_\Gamma$ (resp. $M_{\Gamma,u,k}$) the Alexander matrix associated with the Artin presentation of $\A_\Gamma$,
(resp. the standard presentation of $\A_{\Gamma,u,k}$ given in \S\ref{Subgroups}).
The purpose of this section is to describe some relevant properties of both~$M_\Gamma$ and~$M_{\Gamma,u,k}$.

Among these properties, the most relevant for our purposes refer to their rank. Note that, since these matrices have 
coefficients in a ring of Laurent polynomials~$R=\CC[\ZZ^m]$, a matrix $A\in \Mat(R)$ has rank at least $r$ if and only if 
there is a value $p=(t_1,\dots,t_m)\in \CC^m$ such that $A\otimes R/\mathfrak{m}_p\in \Mat(\CC)$ has an $r\times r$ non-zero minor,
where $\mathfrak{m}_p$ denotes the maximal ideal at~$p$. This operation will be called \emph{evaluating} and will be used
oftentimes to simplify notation.

\begin{lem}
The rank of the Alexander matrix $M_{\Gamma}$ defined above is exactly~$|V|-1$.
\end{lem}

\begin{proof}
Consider $M_T$ the row submatrix of $M_\Gamma$ given by the $|V|-1$ relations determined by the edges of 
a maximal tree $T$ in $\Gamma$. Since $M_T$ clearly has rank $|V|-1$, the matrix $M_{\Gamma}$ has rank at least~$|V|-1$.

To see the equality, consider $\bar \Gamma=(V,\bar E,2\bar \ell)$ the completion of the 
graph $\Gamma$ obtained from $\Gamma$ by adding an edge of label 2 for every pair of disconnected vertices.
The matrix $M_{\bar{\Gamma}}$ associated with this graph contains $M_\Gamma$ as a submatrix.
Choose any vertex $v\in V$, we will show that the $|V|-1$ rows associated to the relations involving $v$ 
generate the remaining rows.

Consider $e=\{w,w'\}\in \bar E$, using Lemma~\ref{lemma-foxA}, the row $f_{e}$ associated with the classical Artin 
relation $A_{\bar{\ell}_{e}}(w,w')$ has the form:
\begin{equation}\label{eq-row} 
p_{\bar\ell_{e}}(t_{w}t_{w'})
\begin{array}{ccccccccccc}
(0 & ...& 0& (1-t_{w'})& 0& ... &0& (t_{w}-1)& 0& ... &0)
\end{array}
\end{equation}
where the non-zero elements are at the columns corresponding to the vertices $w$ and $w'$ respectively.

Note that, since $\bar\Gamma$ is a complete labeled even graph, the three vertices $v,w,w'\in \bar V=V$ form 
a triangle, that is, $e=\{w,w'\}$, $e_1=\{v,w\}$, $e_2=\{v,w'\}$. Moreover, the rows $f_e$, $f_{e_1}$, and $f_{e_2}$ satisfy
the following linear combination:
$$\dfrac{(t_{v}-1)}{p_{\bar\ell_{e}}(t_wt_{w'})}f_{e}+
\dfrac{(t_{w'}-1)}{p_{\bar\ell_{e_1}}(t_vt_{w})}f_{e_1}+
\dfrac{(t_{w}-1)}{p_{\bar\ell_{e_2}}(t_vt_{w'})}f_{e_2}=0.
$$
Thus, $M_{\bar{\Gamma}}$ has rank less than or equal to $|V|-1$. 
Since $M_{\Gamma}$ is a submatrix of $M_{\bar{\Gamma}}$ the result follows.	
\end{proof}

\begin{notation}\label{not-gammauk}
Recall from Theorem~\ref{thm-pres} that the generators of a standard presentation 
of $\A_{\Gamma,u,k}$ can be distinguished in three type groups $\{\bar u\}\cup V_{2,u}\cup W_{k,u}$, 
where 
$$V_{2,u}=\{v\in V\mid e=\{u,v\}\in V, \ell_e=1\}$$
and 
$$W_{k,u}=\{w_{i,j} \mid w_i\in W=V\setminus (\{\bar u\}\cup V_{2,u}), j\in \ZZ_k\}.$$
In the sequel, the elements in $V_{2,u}$ will be denoted by $v_1,...,v_m$,
where $m$ is the number of vertices adjacent to $u$ with label $2$.
Analogously, the elements of $W_{k,u}$ will be denoted by $w_{i,j}$, for 
$w_i\in W$, where $1\leq i\leq n= |V|-m-1$ and $j\in \ZZ_k$.
\end{notation}


From the results of the two previous sections, we immediately obtain the following 
description of the Alexander matrix~$M_{\Gamma,u,k}$.

\begin{lem}
\label{lem-Muk}
The Alexander matrix $M_{\Gamma,u,k}$  of $\A_{\Gamma,u,k}$ associated with 
its standard presentation has the following form:
$$
\begin{blockarray}{ccccccccccccc}
w_{*,0} & & w_{*,1}& & ... & & w_{*,k-1}  & & v_{1} & ... & v_{m} & & \bar u \\
\begin{block}{(ccccccccccccc)}
0 & \vline & 0 & \vline & ... & \vline & 0 & \vline & & A_k & & \vline   & 0 \\ 
\cline{1-13}
A'_0 & \vline & 0 & \vline & ... & \vline & 0 & \vline &  & A_0 &  &  \vline  & 0 \\ 
\cline{1-13}
0 & \vline & A'_1 & \vline & ... & \vline & 0 & \vline &  &A_1 &  &  \vline  & 0\\ 
\cline{1-13}
... & & ...& & ...& &...& & ...&... &... && ...  \\
\cline{1-13}
0 & \vline & 0 & \vline & ... & \vline & A'_{k-1} & \vline &  & A_{k-1} &  &  \vline  & 0 \\ 
\cline{1-13}
&  \vline & & \vline & & \vline &  &  \vline  & t_{\bar u}-1 & & & \vline & 1-t_{v_1}  \\ 
0 & \vline  & 0 & \vline & 0 & \vline & ... & \vline  & & \ddots & & \vline & \vdots\\
& \vline  & & \vline  & & \vline  & & \vline & &  &t_{\bar u}-1 & \vline&  1-t_{v_m}\\
\cline{1-13}
\\ 
& & & & & 
& &
M_{B}
& & & & & \\
\\
\end{block}
\end{blockarray}
$$
where:
\begin{enumerate}
\item 
$w_{*,j}$ denotes the set of columns associated with all the generators of type $w_{i,j}\in W_{k,u}$ 
for a fixed $j\in \ZZ_k$, with $w_i\in W$, as in Notation~\ref{not-gammauk}.
\item $A_k$ is the Alexander matrix corresponding with relations of type $R$\eqref{rel-vv} in 
Theorem~\ref{thm-pres} with respect to the generators $\{v_1,\dots,v_m\}$.
\item the submatrices $A'_j$ and $A_j$ are so that the matrix $(A'_j \ \ \vline\ \ A_j)$
is the Alexander matrix of the relations of type $R$\eqref{rel-vw} in Theorem~\ref{thm-pres}
with respect to the generators $\{w_{*,j},v_1,\dots,v_m\}$ , i.e. their rows are of the form:
$$f_{a,b}\equiv p_{c_{ab}}(t_at_b) \begin{array}{ccccccccccc}
(0 & ...& 0& (1-t_{b})& 0& ... &0& (t_{a}-1)& 0& ... &0)
\end{array}$$
for $a=v_l\in V_{2,u}$ and $b=w_{i,j}\in W_{k,u}$.
\item The submatrix $M_B$ is the Alexander matrix associated with the relations
of type $(RB)$ in Theorem~\ref{thm-pres}. Note that this is a block matrix whose
blocks are the submatrices $M_{B(w,u)}$ associated with the relations of type 
$B^i_{\ell_e,k}(w,\bar u)$, for $i\in \ZZ_k$ and~$\{u,w\}\in E$.
\end{enumerate}
\end{lem}


\begin{lem}\label{PArank}
The submatrix $M_{B(w,u)}$ has maximal rank.
\end{lem}

\begin{proof}
As was mentioned above, we are assuming $\{u,w\}\in E$.
Let us distinguish two cases depending on whether or not $\ell_e$ is a multiple of~$k$.
\begin{enumerate}
\item Assume $\ell_e\equiv 0 \mod k$. Using Lemma~\ref{lemma-foxB} and evaluating
$t_{w,0}=t_{w,1}=...=t_{w,k-2}=1$ in $M_{B(w,u)}$ the following upper triangular matrix
is obtained:
$$
M=\begin{blockarray}{cccccc}
 w_0 & w_1 & ... & w_{k-2} & w_{k-1} & \bar u \\
\begin{block}{(ccccc|c)}
 t_{\bar u}-1 & t_{\bar u}-1 & ... & t_{\bar u}-1 & t_{\bar u}-1 & 1-t_{w,k-1} \\
 1-t_{w,k-1}t_{\bar u} & 0 & ... & 0 & 0 & 0\\
 0 & 1-t_{w,k-1}t_{\bar u} & ... & 0 & 0 & 0\\
 0 & 0 & \ddots & \vdots & \vdots & \vdots \\
0&0&...& 1-t_{w,k-1}t_{\bar u} & 0 & 0 \\
\end{block}
\end{blockarray}
$$		 
which has maximal rank.
\item Assume $\ell_e\not\equiv 0 \mod k$. Write $\ell_e=c_ek+r_e$, with $0<r_e<k$. 
Analogously to the previous case, using Lemma~\ref{lemma-foxB} and evaluating now at 
$t_{w,0}=t_{w,1}=...=t_{w,k-2}=t_{w,k-1}=1$, the following matrix is obtained:
$$
M=\begin{blockarray}{cccccccccc}
w_0 & ... & w_{r-1} & w_{r} & ... & w_{k-r-1} & w_{k-r}& ... & w_{k-1} & \bar u\ \\
\begin{block}{(ccccccccc|c)}
1 & & & -t_{\bar u}^{c} & & & & & &0\\
& \ddots & &  & \ddots & & & & &\vdots\\
&  &\ddots &  &  &\ddots & & & &\vdots\\
&  && \ddots  &  & &\ddots & & &\vdots\\
&  &&   & \ddots & & &\ddots & &\vdots\\
& & &  & & 1 & & & -t_{\bar u}^{c}&\\
\cline{1-10}
-t_{\bar u}^{c+1}  & & &  & &  &1 & & &\vdots\\
& \ddots & &  & & & &\ddots & &\\
& &-t_{\bar u}^{c+1} &  & &  & & &1 &0\\
\end{block}
\end{blockarray}
$$
Formally, $t_{\bar u}=0$ produces a matrix of maximal rank and hence the result follows
using small enough values of $t_{\bar u}$.
\end{enumerate}
\end{proof}

\begin{obs}\label{PAmatrixrep}
Note that, in the previous Lemma, the submatrix of $M_{B(w,\bar u)}$ resulting 
from deleting the column $\bar u$ has a maximal rank.
Therefore, in order to study the rank of $M_{\Gamma,u,k}$, and after row operations, one 
can assume that $M_{B(w,\bar u)}$ is equivalent to:
$$
\begin{blockarray}{cccccc}
 w_{0} & w_{1} & ... & w_{k-2}& w_{k-1} & \bar u\ \\
\begin{block}{(ccccc|c)}
* &* & \hdots & * & * & *\\
 0 & * & \hdots & *& *& *\\
 0 & \ddots & \ddots& * & \vdots& \vdots\\
 0 & 0 & \ddots& \ddots & \vdots & \vdots\\
 0 & 0 & \hdots & 0& * & * \\
\end{block}
\end{blockarray}
$$
\end{obs}

Recall that the corank of a matrix $M$, is defined as 
$$\corank(M)=\#\text{columns}(M)-\rank(M).$$
Then one has the following result on the corank of~$M_{\Gamma,u,k}$.

\begin{lem}\label{rankAM}
Under the conditions above $\corank (M_{\Gamma,u,k})\leq 1$.
\end{lem}

\begin{proof}
Let us consider $\Gamma_{u}=\Gamma\setminus \{u\}$. 
We will first assume that $\Gamma_{u}$ is connected. Following the notation above, 
recall that $V_{2,u}=\{v_1,\dots,v_m\}$ denotes the set of vertices adjacent to $u$ 
with label $2$ and $W=\{w_1,\dots,w_n\}$ denotes the set of remaining vertices of $\Gamma_{u}$. 
We will consider the matrix $M$ obtained eliminating the column corresponding to $u$ from the Alexander matrix $M_{\Gamma,u,k}$ (which has $(nk+m+1)-1=nk+m$ columns). We will prove the result showing that $M$ has maximal rank:
\begin{enumerate}
\item If $n=0$, the matrix $M$ becomes: 
$$
M=\begin{blockarray}{ccc}
v_{1} & \hdots & v_{m} \\
\begin{block}{(ccc)}
1-t_{\bar u} &  &\\
& \ddots & \\
&  & 1-t_{\bar u} \\
* & * & * \\
... & ... & ... \\
* & * & * \\
\end{block}
\end{blockarray}
$$
which has maximal rank.
		
\item If $n\neq 0$, consider $T$ a spanning tree on~$\Gamma_{u}$.
\begin{enumerate}
\item Assume $m \neq 0$. In this case we will describe certain submatrices of $M_\Gamma$ 
which will appear as blocks in $M_{\Gamma,u,k}$ of the appropriate rank.

In order to do this, note that $T$ will contain at least $n$ edges $e_{1},..,e_{n}$ satisfying 
that each $e_i$ involves at least one vertex in $W$ and $W\subset V(\{e_{1},...,e_{n}\})$.
Let us denote by $S\subset T$ the forest containing the edges $e_{1},..,e_{n}$. Note that $S$ 
defines a submatrix $M_0$ of $M_{\Gamma_{u}}$. We will show that columns and rows can be ordered 
in such a way that $M_0$ is upper triangular, every diagonal element is non-zero, and the 
columns associated with the vertices $W$ come first.

This can be easily seen by induction. In case $\Gamma_{u}$ has only two vertices, say $v$ and $w$ 
(this is by hypothesis the minimum number of vertices), and only one edge, the matrix $M_0$ is 
a row matrix of type~\eqref{eq-row} whose columns can be reordered as wanted.
Now, suppose the result is true for $\lambda-1$ vertices and consider now the case when $\Gamma_{u}$ 
has exactly $\lambda$ vertices. Choose a vertex $w'$ in $V(S)$ of degree $1$. Note that, by definition, 
$S$ must contain at least one such vertex in $W$, so one can assume $w'\in W$.
Then $S\setminus \{w'\}$ verifies the result. The matrix $M_0$ results from the latter after adding 
one column (associated with $w'$) and one row $f$ (associated with the edge containing $w'$). 
Note that placing $w'$ as the first column and $f$ as the first row concludes the proof.
			
Also note that the submatrix $M_n$ of $M_0$ resulting from keeping only the columns associated with 
the vertices in $W$ appears as is in $k$ blocks in $M_{\Gamma,u,k}$ corresponding to the copies of 
the vertices in $W$ and the relations associated with the edges of $S$.
This produces a square submatrix $M_k$ of $M_{\Gamma,u,k}$ of size $kn$ and non-zero determinant.
Finally, let us add to $M_k$ the columns associated with all vertices in $V_{2,u}$ placed at the end. 
Since every $v_{i}\in V_{2,u}$ is adjacent to $u$ with label $2$ the relations associated with these 
edges result in rows producing an upper triangular square submatrix $M$ of size $kn+m$ whose determinant 
is non-zero as below.
$$
M=\begin{blockarray}{cccccc}
& W &  &  & V &  \\
\begin{block}{( ccc|ccc)}
M_n & & & & &\\
0& \ddots &&& \\
\vdots& 0  & M_n\\
\cline{1-6}
\vdots& \vdots & \ddots&1-t_{\bar u}&&\\
\vdots& \vdots &\ddots&\ddots&\ddots&\\
0& 0 & 0 &0&0&1-t_{\bar u}\\
\end{block}
\end{blockarray}
$$
This ends this case.

\item If $m=0$, then the spanning tree $T$ consists of $n$ vertices and $n-1$ edges. Let us 
consider $w_n$ a vertex in $\Gamma$ adjacent to $u$ (there must be at least one since $\Gamma$ is connected).
Consider the Alexander submatrix $M_{B(w_n,\bar u)}$ associated with relations of type $(RB)$, which by 
Remark~\ref{PAmatrixrep}, is equivalent to:

$$
\begin{blockarray}{ccccc}
w_{n,0} & w_{n,1} & ... & w_{n,k-2}& w_{n,k-1} \\
\begin{block}{(ccccc)}
* &* & \hdots & * & *\\
0 & * & \hdots & *& *\\
0 & \ddots & \ddots& * & \vdots\\
0 & 0 & \ddots& \ddots & \vdots\\
0 & 0 & \hdots & 0& * \\
\end{block}
\end{blockarray}
$$
Let $f_{1},...,f_{k}$ denote the $k$ rows of this matrix.

On the other hand, let $M_T$ be the $(n-1) \times n$ submatrix $M_{\Gamma}$ associated with $T$.
Let us order the relations in such a way that $M_T$ is upper triangular with non-zero diagonal 
elements and whose last column corresponds to $w_n$ --\, in other words, the vertex associated 
to the first column must have degree $1$. 

For each group of copies of the vertices $w_{j,p}$, there is a copy of the tree $T$ with an 
analogous matrix $M_{T,p}$. Now, one can write the Alexander matrix $M_\Gamma$ in the following 
way: the first rows correspond to the matrix $M_T$, then the row $f_{1}$ completed with zeroes 
where necessary, then the rows corresponding to the matrix $M_{T,1}$, then the row $f_{2}$. 
Finally the matrix $M_{T,k-1}$ and the row $f_{k}$. This matrix is clearly upper triangular 
and it has maximal rank $(kn=kn+m)$.
\end{enumerate}
\end{enumerate}

Summarizing, if $\Gamma_{u}$ is connected, then $\rank (M_{\Gamma,u,k})\geq nk+m$, and hence 
$\corank (M_{\Gamma,u,k})\leq 1$.

Assume now that $\Gamma_{u}$ is not connected, and denote by $\Gamma_{1},...,\Gamma_{s}$ its connected components.

Then, the Alexander matrix $M_{\Gamma,u,k}$ after removing the column $\bar u$ is of the form:

$$
\begin{blockarray}{cccc}
C_{1} & C_{2} & ... & C_{s} \\
\begin{block}{( cccc)}
M_{C_{1}} &0  &\hdots&0\\
0& M_{C_{2}} &\ddots&\vdots \\
\vdots&\ddots&\ddots &0 \\
0&\hdots& 0 & M_{C_{s}}\\ 
\end{block}
\end{blockarray}
$$
where $M_{C_{i}}$ corresponds to a connected graph. 
The result follows from the connected case since the matrix is block-diagonal.
\end{proof}

\begin{lem}\label{AnPA1}
Assume $e=\{w,u\}\in E$ such that $\ell_e\equiv 0\mod k$, then the matrix $M_{B(w,\bar u)}$ 
has rank 1 over $\Lambda/\mathfrak{p}$, where $\mathfrak{p}$ is the ideal generated by 
$1-t_{\bar u}t_{w,0}\cdots t_{w,k-1}$.
\end{lem}

\begin{proof}
By Lemma~\eqref{lemma-foxB} we know that $M_{B(w,\bar u)}$ is a multiple by 
$p_{c_{e}}(t_{\bar u}\bar{t}_{w})$ of the following matrix
$$
M=
\left(\begin{array}{cccc}
1-\bar{t}_{w,0,k} & t_{\bar u}-1 & ... & \bar{t}_{w,0,k-1}(t_{\bar u}-1)\\
\bar{t}_{w,1,k}(t_{w,0}-1) & 1-\bar{t}_{w,1,0} & ... & \bar{t}_{w,1,k-1}(t_{w,0}-1)\\
... & ... & ... & ...\\
\bar{t}_{w,k-1,k}(t_{w,k-2}-1) & \bar{t}_{w,k-1, 0}(t_{w,k-2}-1) &...& t_{w,k-2}-1
\end{array}\right)$$
Note that, mod $\mathfrak{p}$, $M$ can be written as
$$\left(\begin{array}{cccc}
t_{\bar u}^{-1}(t_0-1) & t_{\bar u}-1 & ... & \bar{t}_{w,0,k-1}(t_{\bar u}-1)\\
t_{\bar u}^{-1}t_{w,0}^{-1}(t_{w,0}-1) & t_{w,0}^{-1}(t_{w,0}-1) & ... & \bar{t}_{w,1,k-1}(t_{w,0}-1)\\
... & ... & ... & ...\\
t_{\bar u}^{-1}t_{w,0}^{-1}...t_{w,k-2}^{-1}(t_{w,k-2}-1) & t_{w,0}^{-1}...t_{w,k-2}^{-1}(t_{w,k-2}-1) 
&...& t_{w,k-2}-1
\end{array}\right)$$
If $f_{j}$ denotes the $j$-th row of $M$ note that
$$(t_{\bar u}-1)t_{w,0}\cdots t_{w,j-2}f_j=(t_{w,j-2}-1)f_1$$
for any $j=2,...,k-1$ and thus the result follows.
\end{proof}


\section{$\QP$-Irreducible graphs}
\label{sec-irreduicible}
As mentioned in the Introduction, a graph is called \emph{quasi-projective} --\,or 
$\QP$-graph\,-- if its associated Artin group is in $\G$. The purpose of this section 
is to describe the simplest $\QP$-graphs, referred to as $\QP$-\emph{irreducible} graphs for
even Artin groups.

\begin{dfn}\label{BasicPieces}
We call $\Gamma$ a $\QP$-\textbf{irreducible graph} 
if $\A_\Gamma$ is quasi-projective and it cannot be obtained as a 2-join
of two quasi-projective graphs.
\end{dfn}

By Proposition~\ref{prop:QP}\eqref{QP2J}, the 2-join of $\QP$-graphs must be a $\QP$-graph. However,
in general, properties on Artin groups are not easily read from subgraphs. This result 
allows one to read an obstruction to quasi-projectivity from certain subgraphs to graphs.

\begin{dfn}
We say that $\Gamma_1$ is a \textbf{$v$-subgraph} of $\Gamma$ if $\Gamma_1$ is obtained 
from $\Gamma$ by deleting some vertices. We will denote it $\Gamma_1\subset_v \Gamma$. 
In this situation $\Gamma$ is called a \textbf{$v$-supergraph} of~$\Gamma_1$.
\end{dfn}

\begin{lem}\label{subgrafo}
Let $\A_{\Gamma_1}$ be the Artin group of $\Gamma_1=(V_1,E_1,m_1)$. Assume that for certain 
$k\in \mathbb Z_{\geq 2}$, and $u\in \Gamma_1$, the subgroup 
$\hat G_k:=\A_{\Gamma_1,u,k}\subset \A_{\Gamma_1}$ 
satisfies that there exist two ideals 
$\hat I_{1},\hat I_{2}\subset \hat \Lambda_k:=\mathbb C[H_1(\hat G_k)]$ such that:
\begin{enumerate}
\enet{\rm (C\arabic{enumi})}
\item\label{subgrafo-prop-1} $Z(\hat I_i)\subset V_{r_i}(\hat G_k)$, $r_i\geq 1$ for $i=1,2$,
\item\label{subgrafo-prop-3} $\dim (Z(\hat I_{1}+\hat I_{2}))\geq 1$, and 
\item 
\begin{enumerate}
\eneti{\rm (\alph{enumii})}
 \item\label{subgrafo-prop-2} either $Z(\hat I_1+\hat I_2)\subset V_r(\hat G_k)$ for $r>\max\{r_1,r_2\}$,
 \item\label{subgrafo-prop-2b} or $\hat I_1, \hat I_2$ are prime ideals of~$\hat \Lambda_k$. 
\end{enumerate}
\end{enumerate}
Then $\A_{\Gamma_1}$ is not quasi-projective.
	
Moreover, if $\Gamma=(V,E,m)$ is any $v$-supergraph of $\Gamma_1$ such that 
$m_e$ is even for any $e=\{v,w\}\in E$, $v\in V_1$, $w\in V\setminus V_1$, 
then $\A_{\Gamma}$ is not quasi-projective.
\end{lem}

\begin{proof}
Let us prove the first part by contradiction. 
Assuming that $\A_{\Gamma_1}$ is quasi-projective would imply that the co-cyclic group $\A_{\Gamma_1,u,k}$
is also quasi-projective by Proposition~\ref{prop:QP}\eqref{QPGL}. The strategy of this proof is to 
reach a contradiction on the quasi-projectivity of $\A_{\Gamma_1,u,k}$ by finding two irreducible 
components of its characteristic variety intersecting in a positive dimensional component and
thus contradicting Proposition~\ref{Fin}.
Let us assume that $r_1\geq r_2$. Note that the set of zeroes $Z(\hat I_i)$ may be non-irreducible,
but, using condition~\ref{subgrafo-prop-2} in the statement, there exists an irreducible 
component, say $H_1$ (resp. $H_2$) in $Z(\hat I_1)$ (resp. $Z(\hat I_2)$) which is not contained in 
$Z(\hat I_2)$ (resp. in $Z(\hat I_1)$). 
By condition~\ref{subgrafo-prop-3} their intersection $H_1\cap H_2$ has dimension greater 
or equal to $1$. 

To prove the \textit{moreover} part, we will show that $\A_{\Gamma}$ also satisfies the hypotheses
of the first part, that is, that there exist two ideals $I_1, I_2\subset \Lambda_k:=\mathbb C[H_1(G_k)]$ 
satisfying conditions~\ref{subgrafo-prop-1}-\ref{subgrafo-prop-3} and either~\ref{subgrafo-prop-2}
or~\ref{subgrafo-prop-2b} for the subgroup 
$G_k:=\A_{\Gamma,u,k}\subset \A_{\Gamma}$. Note that, the condition on the parity of the labels
joining vertices from $V_1$ and $V\setminus V_1$ ensures the existence of a commutative diagram
$$
\array{ccccccccc}
1 & \to & G_k & \to & G_\Gamma & \to & \ZZ & \to & 1 \\
& & & & \downarrow & & || & &\\
1 & \to & \hat G_k & \to & G_{\Gamma_1} & \to & \ZZ & \to & 1 \\
\endarray
$$
which allows for the existence of a morphism $H_1(G_k)\to H_1(\hat G_k)$ extending to 
$\Lambda_k\to \hat\Lambda_k$. Moreover, $\hat\Lambda_k=\Lambda_k/I$ for a certain ideal.
In order to describe it let us decompose $V$ as a disjoint union 
$V=V_1\cup \tilde V_{2,u}\cup W$, where 
$\tilde V_{2,u}=\{v\in V\setminus V_1\mid e=\{u,v\}\in E, m_e=2\}$. Then 
$$I=\text{Ideal }\left(\{t_v-1 \mid v\in V\setminus \tilde V_{2,u}\} \cup \{t_{w,j}-1 \mid w\in W, j\in \ZZ_k\}\right).$$

Since the tensor product is right exact, the matrix $\hat M_{\Gamma,u,k}=M_{\Gamma,u,k} \otimes \Lambda/I$ 
determines the Alexander $\hat \Lambda_k$-module of~$\hat G_k$.
We claim that 
\begin{equation}\label{eq-Mhat}
\hat M_{\Gamma,u,k}=\left(\begin{array}{c|c}
M_{\Gamma_1,u,k} & 0 \\
\hline
0 & A'
\end{array}\right).
\end{equation}
In order to check this, first note that the submatrix of $\hat M_{\Gamma,u,k}$ whose rows are 
associated to the edges of $\Gamma_{1}$ has the form
$$\left(\begin{array}{c|c}
M_{\Gamma_1,u,k} & 0 
\end{array}\right).$$
The claim will follow if we prove that the remaining rows, associated with the
edges in $E\setminus E_1$, satisfy that any entry in a column in~$V_1$ is in $\hat I$.
The latter is a consequence of~\eqref{eq-row} and Lemma~\ref{lemma-foxB}.

Finally, note that if condition~\ref{subgrafo-prop-2} (resp.~\ref{subgrafo-prop-2b}) is satisfied for $\hat I_i$, 
then also condition~\ref{subgrafo-prop-2} (resp.~\ref{subgrafo-prop-2b}) is satisfied for $I_i=I+\hat I_i$
using~\eqref{eq-Mhat} (resp.~using that $Z(I_i)=Z(\hat I_i)\times \{1\}$ is irreducible). 
Therefore the ideals $I_1, I_2\subset \Lambda_k$ also satisfy the conditions of the statement 
for~$\A_\Gamma$ and the result follows.
\end{proof}

\begin{obs}
By Theorem~\ref{Suciu2}, the only $\QP$-irreducible right-angled graphs are sets of $r$ disconnected 
vertices,~$\overline{K}_r$. On the other hand, we have established by Theorem~\ref{BP} that both the 
segment $S_{2\ell}$ with label $2\ell$ ($\ell>1$) and the triangle $T(4,4,2)$ are also $\QP$-irreducible graphs. 
\end{obs}

\begin{center}
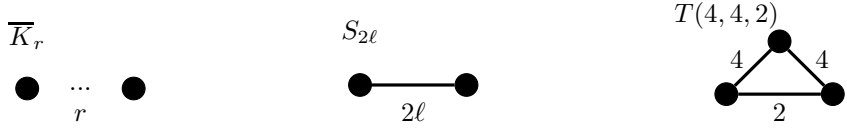
\begin{figure}[ht]
\begin{tabular}{ccccc}
\begin{tikzpicture}[scale=.7,vertice/.style={draw,circle,fill,minimum size=0.3cm,inner sep=0}]
\node[vertice] at (1,1) {};
\node[vertice] at (-1,1) {};
\node at (0,1) {$...$};
\node at (0,.5) {$r$};
\node at (-1,2) {$\overline{K}_r$};
\end{tikzpicture}
& \quad \quad  \quad \quad \quad \quad&
\begin{tikzpicture}[scale=.7,vertice/.style={draw,circle,fill,minimum size=0.3cm,inner sep=0}]
\node[vertice] (D1) at (1,1) {};
\node[vertice] (F1) at (-1,1) {};
\draw[very thick] (D1)--(F1);
\node at (0,.5) {$2\ell$};
\node at (-1,2) {$S_{2\ell}$};
\end{tikzpicture}
& \quad \quad  \quad \quad \quad \quad&
\begin{tikzpicture}[scale=.7,vertice/.style={draw,circle,fill,minimum size=0.3cm,inner sep=0}]
\node[vertice] (D1) at (1,.5) {};
\node[vertice] (E1) at (0,1.5) {};
\node[vertice] (F1) at (-1,.5) {};
\draw[very thick] (D1)--(F1)--(E1)--(D1);
\node[left] at (-0.5,1.15) {$4$};
\node[right] at (0.5,1.15) {$4$};
\node[below] at (0,.5) {$2$};
\node at (-1,2) {$T(4,4,2)$};
\end{tikzpicture}
\end{tabular}
\caption{$\QP$-irreducible graphs of type $\overline{K}_r$, $S_{2\ell}$, and $T(4,4,2)$.}
\label{fig-simple-blocks2}
\end{figure}
\end{center}

The purpose of this section is to show that the only $\QP$-irreducible graphs are 
$\overline{K}_r$, $S_{2\ell}$ $(\ell>1)$, and $T(4,4,2)$.

First we can assume that our graph has at least three vertices, otherwise it is 
$\QP$-irreducible if and only if it is disconnected $\overline{K}_r$ ($r=1,2$) or a 
segment $S_{2\ell}$ $(\ell>1)$. The second reduction is given in~\cite{ACM-artin}, 
for \emph{strictly even} graphs, that is, even graphs that are not right-angled.
We recall it here.

\begin{thm}[{\cite[Thm. 5.26]{ACM-artin}}]
\label{thm-complete}
If $\Gamma$ is a strictly even, non-complete graph with at least three vertices, 
then $A_{\Gamma}$ is not quasi-projective.
\end{thm}

This result is shown by proving that the characteristic varieties of the Artin groups 
of non-complete strictly even graphs contain two irreducible components having a 
positive dimensional intersection, which contradicts Proposition~\ref{Fin}. 

Note that $\QP$-irreducible even graphs --\,other than a point\,-- are necessarily strictly 
even. Hence, the purpose of the rest of the section is to study complete $\QP$-irreducible 
graphs of three or more vertices.

\subsection{Complete $\text{QP}$-irreducible graphs with 3 vertices}

\begin{thm}\label{Tri1}
Assume $\Gamma$ is an even $v$-supergraph of $T(2r,2k,2\ell)$ with $r\geq 3$ and $k\geq 2$ 
--\,see Figure~\ref{fig-teo31}. Then $\A_{\Gamma}$ is not quasi-projective.
\begin{center}
\begin{figure}[h!]
\begin{tikzpicture}[vertice/.style={draw,circle,fill,minimum size=0.3cm,inner sep=0}]
\node[vertice] (D1) at (1,.5) {$w$};
\node[vertice] (E1) at (0,1.5) {$u$};
\node[vertice] (F1) at (-1,.5) {$v$};
\draw[very thick] (D1)--(F1)--(E1)--(D1);
\node[left] at (-0.5,1.15) {$2r$};
\node[right] at (0.5,1.15) {$2k$};
\node[below] at (0,.5) {$2\ell$};
\node[left] at (-1.3,.5) {$v$};
\node[right] at (1.3,.5) {$w$};
\node[above] at (0,1.8) {$u$};
\end{tikzpicture}
\caption{}
\label{fig-teo31}
\end{figure}
\end{center}
\end{thm}

\begin{proof}
Without loss of generality, one can assume $r\geq k \geq \ell$. 
Four separate cases will be considered. The first case will be shown in detail.
The remaining cases follow analogously.
\begin{enumerate}
\item In case $r\geq 4$, $k\geq 2$, one can consider the index $r$ subgroup 
$\A_{T,u,r}\subset \A_{T}$, where $T=T(2r,2k,2\ell)$. 
According to Lemma~\ref{lem-Muk} $M_{T,u,r}$ has two $B$-Artin submatrices 
$M_{B(v,\bar u)}$ and $M_{B(w,\bar u)}$ of $r$ rows each and an Artin submatrix of $r$ 
rows (corresponding to the $r$ relations $A_{\ell}(v_i,w_i)$ for $i\in \ZZ_r$.
Hence $M_{T,u,r}$ is a $3r\times (2r+1)$ matrix whose corank is $\leq 1$ by Lemma~\ref{rankAM}.		
Let us define $p=1-t_{\bar u}\bar t_{v}$ and consider the following ideals
\begin{align*}
& I_{1}=(p,t_{v,0}-1,t_{v,1}-1,t_{w,0}-1,t_{w,1}-1)\\
& I_{2}=(p,t_{v,0}-1,t_{v,2}-1,t_{w,0}-1,t_{w,2}-1).
\end{align*}
Note that $\rank (M_{B(v,\bar u)}|_{I_{i}})=1$ by Lemma~\ref{AnPA1}. In addition, note that the 
first two rows of $M_{T,u,r}|_{I_{1}}$ are zero and also the first and third rows of $M_{T,u,r}|_{I_{2}}$.
Summarizing, $M_{T,u,r}|_{I_{i}}$ contains three submatrices $M_{i,A}$, $M_{i,B(v,\bar u)}$, and $M_{i,B(w,\bar u)}$,
where $\rank (M_{i,A})\leq r-2$, $\rank (M_{i,B(v,\bar u)})=1$, and $\rank (M_{i,B(w,\bar u)})\leq r$. Therefore
$\rank (M_{T,u,r}|_{I_{i}})\leq 2r-1$. Since $M_{T,u,r}$ has $2r+1$ columns, one has
$$\corank (M_{T,u,r}|_{I_{i}})\geq 2 >
\corank M_{T,u,r}.$$
Also
$$\corank (M_{T,u,r}|_{I_{1}+I_{2}})< 
\max\{\corank M_{T,u,r}|_{I_{1}},\corank M_{T,u,r}|_{I_{2}}\},$$
since $I_1+I_2$ makes one extra row vanish, which is originally independent from the others.
Finally, $Z(I_1+I_2)$ has dimension $\geq 1$ since the variable $t_{\bar u}$ is free 
(since $r\geq 4$, $P$ gives a relation between $t_{v,3}$ and $t_{\bar u}$ but does not fix any of them).
Therefore, by Lemma~\ref{subgrafo}, $\A_{\Gamma}$ is not quasi-projective.

\item Case $r=3$ and $k=3$ can be treated by considering 
$\A_{T,u,3}\subset \A_{T}$ and the ideals:
$$
\array{l}
I_{1}=(t_{\bar u}-1,t_{v,1}-1,t_{v,2}-1,t_{w,0}-1,t_{w,1}-1)\\
I_{2}=(t_{\bar u}-1,t_{v,1}-1,t_{v,2}-1,t_{w,0}-1,t_{w,2}-1).
\endarray
$$

\item Case $r=3$, $k=\ell=2$ follows after considering 
the subgroup $\A_{T,u,2}\subset \A_{T}$ of index $2$ and the ideals:
$$
\array{l}
I_{1}=(1-t_{\bar u}\bar t_w,1-t_{\bar u}\bar t_v,p_0)\\
I_{2}=(1-t_{\bar u}\bar t_w,1-t_{\bar u}\bar t_v,p_1),
\endarray
$$
with $p_i=1+t_{w,i}t_{v,i}+t_{w,i}^2t_{v,i}^2$, $i=0,1$.

\item Finally, if $r=3$, $k=2$, and $\ell=1$, the result follows after
considering $\A_{T,v,3}\subset\A_{T}$ the subgroup of index 3
and the following ideals:
$$
\array{l}
I_{1}=(p_0,p_1,1-t_{\bar v}\bar t_u)\\
I_{2}=(p_0,p_2,1-t_{\bar v}\bar t_u),
\endarray
$$
where $p_i=1+t_wt_{u,i}$, $i=0,1,2$.
\end{enumerate}
\end{proof}

\begin{thm}\label{Tri2}
Assume $\Gamma$ is an even $v$-supergraph of $T(4,4,4)$. Then, $\A_{\Gamma}$ is not quasi-projective.
\end{thm}

\begin{proof}
Consider $T=T(4,4,4)$ with vertices $V=\{u,v,w\}$ and the index 2 co-cyclic subgroup $\A_{T,u,2}\subset \A_{T}$. 
According to Lemma~\ref{lem-Muk} the Alexander matrix of the associated group is
$$\displaystyle
M_{T,u,2}=\begin{blockarray}{ccccc}
  v_{0}& w_{0} & v_{1}  &w_{1} & \bar u \\
\begin{block}{( ccccc)}
 p_{0}(1-t_{w,0}) & p_{0}(t_{v,0}-1) & 0 & 0 & 0 \\
 0 & 0 & p_{1}(1-t_{w,1}) & p_{1}(t_{v,1}-1) & 0\\
 t_{\bar u}-1 & 0 & t_{v,0}(t_{\bar u}-1) & 0 & 1-t_{v,0}t_{v,1}  \\
 1-t_{v,1}t_{\bar u} & 0 & t_{v,0}-1 & 0 & t_{v,1}(t_{v,0}-1) \\
 0 & t_{\bar u}-1 & 0 & t_{w,0}(t_{\bar u}-1) & 1-t_{w,0}t_{w,1}  \\
 0 & 1-t_{w,1}t_{\bar u} & 0 & t_{w,0}-1 & t_{w,1}(t_{w,0}-1)  \\
\end{block}	
\end{blockarray}
$$
with $p_i=1+t_{v,i}t_{w,i}$, $i=0,1$.
By Lemma~\ref{rankAM}, $M_{T,u,2}$ has corank $\leq 1$. Consider the ideals	
\begin{align*}
& I_{1}=(1-t_{\bar u}\bar t_v,1-t_{\bar u}\bar t_w,p_0)\\
& I_{2}=(1-t_{\bar u}\bar t_v,1-t_{\bar u}\bar t_w,p_1).\\
\end{align*}
By Lemma~\ref{AnPA1}, it is clearly seen that $M_{T,u,2}|I_{i}$ has corank $2$. Therefore
$$2=\max\{\corank M_{T,u,2}|_{I_{1}},\corank M_{T,u,2}|_{I_{2}}\}>1\geq\corank M_{T,u,2}$$
It is also easy to see that $M_{T,u,2}|_{I_1+I_2}$ has corank $3$, which implies
$$\corank M_{T,u,2}|_{I_{1}+I_{2}}=3 > \max\{\corank M_{T,u,2}|_{I_{1}},\corank M_{T,u,2}|_{I_{2}}\}=2.$$
Moreover, $Z(I_1+I_2)$ has dimension $\geq 1$ since $I_1+I_2$ is generated by four polynomials in five variables.
Therefore, by Lemma~\ref{subgrafo}, $\A_{\Gamma}$ is not quasi-projective.
\end{proof}

The previous results combined prove the following.

\begin{cor}
\label{cor-qp}
The only strictly even complete $\QP$-graphs with three vertices 
are $T(2\ell,2,2)$ with $\ell\geq 2$ and~$T(4,4,2)$. Moreover, the latter is
the only $\QP$-irreducible even graph with three vertices.
\end{cor}

\subsection{$\text{QP}$-irreducible even graphs with 4 vertices}
As an immediate consequence of Theorems~\ref{Tri1},~\ref{Tri2}, and~\ref{thm-complete}, the only candidates to 
$\QP$-irreducible even graphs with $4$ vertices must be complete even $v$-supergraphs of either 
$T(2\ell,2,2)$ or~$T(4,4,2)$. Figure~\ref{fig-teo5} shows all possible such graphs.

\begin{center}
\begin{figure}[h!]

\subfigure[{}]{\makebox[.45\textwidth]{
\begin{tikzpicture}[scale=.7,vertice/.style={draw,circle,fill,minimum size=0.3cm,inner sep=-1}]
\node[vertice,label=left:$w_1$] (F1) at (-3,0) {};
\node[vertice,label=right:$w_2$] (E2) at (3,0) {};
\node[vertice,label=above left:$w_3$] (E1) at (0,1) {};
\node[vertice,label=above:$u$] (D1) at (0,4) {};
\draw[very thick] (D1)--(F1)--(E1)--(E2)--(D1);
\draw[very thick] (D1)--(E1);
\draw[very thick] (F1)--(E2);
\node[left] at (-1.7,2) {$4$};
\node[right] at (1.7,2) {$4$};
\node[left] at (-1.2,.8) {$2$};
\node[right] at (1.2,.8) {$2$};
\node[below] at (0,0) {$2$};
\node[right] at (0,2.3) {$4$};
\end{tikzpicture}
}
\label{fig:dibujo1}
}
\hfill
\subfigure[{}]{
\begin{tikzpicture}[scale=.7,vertice/.style={draw,circle,fill,minimum size=0.3cm,inner sep=-1}]
\node[vertice,label=left:$v$] (F1) at (-3,0) {};
\node[vertice,label=right:$w_1$] (E2) at (3,0) {};
\node[vertice,label=above left:$w_2$] (E1) at (0,1) {};
\node[vertice,label=above:$u$] (D1) at (0,4) {};
\draw[very thick] (D1)--(F1)--(E1)--(E2)--(D1);
\draw[very thick] (D1)--(E1);
\draw[very thick] (F1)--(E2);
\node[left] at (-1.7,2) {$2$};
\node[right] at (1.7,2) {$4$};
\node[left] at (-1.2,.8) {$2$};
\node[right] at (1.2,.8) {$2$};
\node[below] at (0,0) {$4$};
\node[right] at (0,2.3) {$4$};
\end{tikzpicture}
\label{fig:dibujo2}
}
\subfigure[{}]{
\begin{tikzpicture}[scale=.7,vertice/.style={draw,circle,fill,minimum size=0.3cm,inner sep=-1}]
\node[vertice,label=left:$w_1$] (F1) at (-3,0) {};
\node[vertice,label=right:$w_2$] (E2) at (3,0) {};
\node[vertice,label=above left:$v$] (E1) at (0,1) {};
\node[vertice,label=above:$u$] (D1) at (0,4) {};
\draw[very thick] (D1)--(F1)--(E1)--(E2)--(D1);
\draw[very thick] (D1)--(E1);
\draw[very thick] (F1)--(E2);
\node[left] at (-1.7,2) {$4$};
\node[right] at (1.7,2) {$4$};
\node[left] at (-1.2,.8) {$4$};
\node[right] at (1.2,.8) {$4$};
\node[below] at (0,0) {$2$};
\node[right] at (0,2.4) {$2$};
\end{tikzpicture}
\label{fig:dibujo3}
}
\caption{}
\label{fig-teo5}
\end{figure}
\end{center}

This list can easily be obtained using the following observation.

\begin{lem}
Any $\QP$-irreducible even graph with at least 3 vertices has labels no larger than~$4$.
\end{lem}

\begin{proof}
By Theorem~\ref{thm-complete} one can assume the graph $\Gamma$ is complete.
Assume $m_e\geq 6$ for some edge $e\in E$ of $\Gamma$. 
By Theorem~\ref{Tri1} all edges adjacent to $e$ must have a label $2$. 
Since $\Gamma$ is complete $\Gamma=\{e\}*_2 \Gamma'$, where $\Gamma'$ is the 
resulting $v$-subgraph after deleting the vertices of~$e$.
\end{proof}

Note that the 4-graph in Figure~\ref{fig:dibujo1} is the only candidate containing 
$T(2,2,2)$, Figure~\ref{fig:dibujo2} is the only candidate containing $T(4,4,2)$, but no $T(2,2,2)$, 
and Figure~\ref{fig:dibujo3} is the only candidate containing $T(4,4,2)$ but no~$T(2\ell,2,2)$.

We are going to see that these three candidates cannot be $\QP$-irreducible graphs. 

\begin{thm}\label{thm-qirred4}
There are no $\QP$-irreducible even graphs of four vertices.

Moreover, an even graph containing any of the graphs in Figure~\ref{fig-teo5}
as a $v$-subgraph is not quasi-projective.
\end{thm}

\begin{proof}
As discussed above, one only needs to rule out the list of graphs shown in~\ref{fig-teo5}.
We will do this separately and using similar arguments. For this reason we will show only 
case~\ref{fig:dibujo1} in detail.
\begin{itemize}
 \item For graph~\ref{fig:dibujo1}
let us consider the index 2 subgroup $\A_{\Gamma,u,2}\subset \A_\Gamma$. 
Its Alexander matrix is given as follows:

\vspace{0.5 cm}

\resizebox{\linewidth}{!}{
$\displaystyle
M_{\Gamma,u,2}=\begin{blockarray}{ccccccc}
  w_{1,0}& w_{2,0} & w_{3,0} & w_{1,1}  &w_{2,1} & w_{3,1} & \bar u   \\
\begin{block}{( ccccccc)}
 1-t_{2,0} & t_{1,0}-1 & 0  & 0 & 0 & 0 & 0  \\
 1-t_{3,0} & 0 & t_{1,0}-1 & 0 & 0 & 0 & 0 \\
 0 & 1-t_{3,0} & t_{2,0}-1 & 0 & 0 & 0 & 0  \\
 0 & 0 & 0 & 1-t_{2,1} & t_{1,1}-1 & 0 & 0 \\
0 & 0 & 0 & 1-t_{3,1} & 0 & t_{1,1}-1 & 0  \\
 0 & 0 & 0 & 0& 1-t_{3,1} & t_{2,1}-1 & 0  \\
 t_{\bar u}-1 & 0 & 0 & t_{1,0}(t_{\bar u}-1) & 0 & 0 & 1-t_{1,0}t_{1,1}  \\
 1-t_{1,1}t_{\bar u} & 0 & 0 & 1-t_{1,0} & 0 & 0 & t_{1,1}(t_{1,0}-1) \\
 0 & t_{\bar u}-1 & 0 & 0 & t_{2,0}(t_{\bar u}-1) & 0 & 1-t_{2,0}t_{2,1} \\
 0 & 1-t_{2,1}t_{\bar u} & 0 & 0 & t_{2,0}-1 & 0 & t_{2,1}(t_{2,0}-1)  \\
 0 & 0 & t_{\bar u}-1 & 0 & 0 & t_{3,0}(t_{\bar u}-1) & 1-t_{3,0}t_{3,1}  \\
 0 & 0 & 1-t_{3,1}t_{\bar u} & 0 & 0 & t_{3,0}-1 & t_{3,1}(t_{3,0}-1)  \\
\end{block}	
\end{blockarray}
$
}

\vspace{0.5 cm}
where $t_{i,j}$ denotes $t_{w_{i,j}}$. 
We now consider the following prime ideals:

\begin{align*}
& I_{1}=(t_{\bar u}-1,t_{1,1}-1,t_{2,0}-1,t_{2,1}-1, t_{3,0}-1)\\
& I_{2}=(t_{\bar u}-1,t_{1,0}-1,t_{1,1}-1,t_{2,1}-1,t_{3,0}-1).\\
\end{align*}
Note that $\corank (M_{\Gamma,u,2}|_{I_1})=\corank (M_{\Gamma,u,2}|_{I_2})=2$,
$\corank (M_{\Gamma,u,2}|_{I_1+I_2})=4$ and 
$$
Z(I_1+I_2)=\{(t_{\bar u},t_{1,0},t_{1,1},t_{2,0},t_{2,1},t_{3,0},t_{3,1})=
(1,1,1,1,1,1,\lambda)\mid \lambda\in \CC^*\}\subset (\CC^*)^7.
$$
The result follows from Lemma~\ref{subgrafo} and the fact that $\dim Z(I_1+I_2)= 1$.

 \item For graph~\ref{fig:dibujo2} consider $\A_{\Gamma,u,2}\subset \A_\Gamma$ and the prime ideals 
$$I_{1}=(t_v-1,t_{\bar u}-1,t_{1,1}-1,t_{2,0}-1), \quad I_{2}=(t_{\bar u}-1,t_{1,1}-1,t_{2,0}-1,1+t_{1,0}t_v).$$
 
 \item For graph~\ref{fig:dibujo3} the result follows considering the subgroup 
$\A_{\Gamma,u,2}\subset \A_\Gamma$ and the prime ideals
$$I_{1}=(t_{\bar u}-1,t_{1,0}-1,t_{2,1}-1,1+t_{1,1}t_v),\quad I_{2}=(t_{\bar u}-1,t_{1,0}-1,t_{2,1}-1,1+t_{2,0}t_v).$$
\end{itemize}
\end{proof}  
 
\subsection{$\text{QP}$-irreducible even graphs with more than 4 vertices}
As a consequence of the results obtained in the previous sections, 
no quasi-projective even Artin group can contain any of the following
\begin{enumerate}
\item\label{cond-1} 
A vertex with two edges with labels $2r$, $r\geq 3$ and $2k$, $k\geq 2$ 
--\,see Theorems~\ref{Tri1} and \ref{thm-complete}.
\item\label{cond-2} A triangle $T(4,4,4)$ --\,see Theorem~\ref{Tri2}.
\item\label{cond-3} A three-edge tree of labels $4,4,4$ 
--\,see Theorems~\ref{thm-complete},~\ref{Tri2}, and \ref{thm-qirred4}.
\end{enumerate}

\begin{thm}
\label{thm-qp2}
There are no $\QP$-irreducible even graphs with more than $3$ vertices.
\end{thm}

\begin{proof}
The result follows for graphs with four vertices by the previous section.

For any $\QP$-irreducible even graph $\Gamma$ with more than four vertices 
note the following: 
\begin{itemize}
 \item $\Gamma$ must be complete by Theorem~\ref{thm-complete}.
 \item If $\Gamma$ contains an edge $e$ with label $m_e=2r$, $r\geq 3$, then 
by~\eqref{cond-1} above, $\Gamma=\{e\}*_2 \hat \Gamma$ and hence $\Gamma$ is not
$\QP$-irreducible.
 \item If $\Gamma$ contains an edge $e$ with label $m_e=4$, then either 
$\Gamma=\{e\}*_2 \hat \Gamma$ (see~\eqref{cond-2} above) or 
$\Gamma=T(4,4,2) *_2 \hat \Gamma$ (see~\eqref{cond-3}).
\end{itemize}
\end{proof}

\section{Proofs of Main Theorems}
\label{sec-proofs}

\subsection{Proof of Theorem~\ref{thm-qp}}
As a consequence of Theorems~\ref{Suciu2} and \ref{BP} graphs of type $\overline K_{r}$, $S_{2\ell}$, and $T(4,4,2)$
as in Figure~\ref{fig-simple-blocks2} are $\QP$-irreducible. Moreover, by Corollary~\ref{cor-qp} and Theorem~\ref{thm-qp2}
these are the only ones. Using~\eqref{eq-2join} and Proposition~\ref{prop:QP}\eqref{QP2J} any 2-join $\QP$-irreducible graphs is 
quasi-projective. This completes the first part of the proof.

For the \emph{moreover} part it is enough to check that the product of two fundamental groups of curve complements
is also the fundamental group of a curve complement. This is a consequence of the following result.

\begin{thm}\label{Oka}\cite{Oka}
Let $C_1$ and $C_2$ be plane algebraic curves in $\mathbb{C}^{2}$. 
Assume that the intersection $C_1 \cap C_2$ consist of distinct $d_1d_2$ points where $d_i$ ($i=1,2$) 
are respective degrees of $C_1$ and $C_2$. Then the fundamental group $\pi_1(\mathbb{C}^2\setminus (C_1\cup C_2))$ 
is isomorphic to the product of $\pi_1(\mathbb{C}^2\setminus C_1)$ and $\pi_1(\mathbb{C}^2\setminus C_2)$.
\end{thm}

\subsection{Proof of Theorem~\ref{thm-kpi1}}
\label{sec-proofkpi1}
Since the product of two $K(\pi,1)$ spaces is also a $K(\pi,1)$ space it is enough to prove the result
for the $\QP$-irreducible even graphs $\overline{K}_r$, $S_{2\ell}$, and $T(4,4,2)$. 
The graph $\overline{K}_r$ is associated with the free 
group $\FF_r$ of rank $r$, which can be realized as the fundamental group of the complement to
$r$ points in $\CC$, which is an Eilenberg-MacLane space.

The group $\A_{S_{2\ell}}$ associated with the segment graph $S_{2\ell}$ is the fundamental group of the complement 
$X$ to the affine curve $\{y-x^\ell\}\cup \{y=0\}$. In projective coordinates $X$ can be seen as the complement
to the projective curve $\cC=\{yz(yz^{\ell-1}-x^\ell)=0\}\subset \PP^2$, that is, $X=\PP^2\setminus \cC$. 
Consider the projection $\pi:\PP^2\setminus \{[1:0:0]\}\to \PP^1$, defined by $[x:y:z]\mapsto [y:z]$. 
Note that $\pi|_X:\to \PP^1\setminus \{[0:1],[1:0]\}$ is well defined, locally trivial fibration and moreover, 
the fiber at each point $[y:z]$ is homeomorphic to $\CC\setminus \{\ell \text{ points}\}$. Thus $X$ is also
an Eilenberg-MacLane space.

Finally, the triangle Artin group $\A_T$ associated with the triangle $T=T(4,4,2)$ is the fundamental group of the 
complement $X$ to the affine curve $\{y-x^2\}\cup \{2x-y-1=0\}\cup \{2x+y+1=0\}$ \cite[Example 5.11]{ACM-artin}. 
Using the identification $\CC^2=\PP^2\setminus \{z=0\}$ we can think of $X$ as the complement of a smooth conic and 
three tangent lines in the complex projective plane. After an appropriate change of coordinates, $X$ can be given as 
$\PP^2\setminus \cC$, where $\cC=\{F(x,y,x)=xyz(x^2+y^2+z^2-2(xy+xz+yz))=0\}\subset \PP^2$. 

\begin{center}
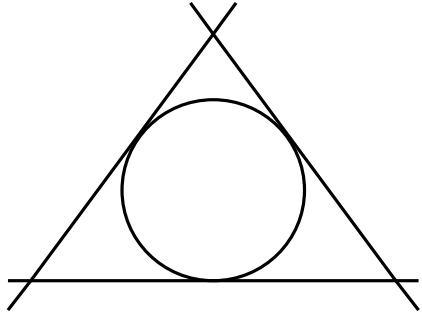
\begin{figure}[ht]
\begin{tikzpicture}[scale=.6]
\draw[very thick] (0,-0.05) circle (2cm);
\draw[very thick] (-4.5,-2.7) -- (-4,-2.05) -- (0,3.4) -- (.5,4.09);
\draw[very thick] (-4.5,-2.05) -- (4.5,-2.05);
\draw[very thick] (4.5,-2.7) -- (4,-2.05) -- (0,3.4) -- (-.5,4.09);
\end{tikzpicture}
\caption{Projective curve $\cC=\{F=0\}$}
\label{fig-ceva}
\end{figure}
\end{center}

We will consider a 4-fold cover $X_4$ of $X$.
Since the higher homotopy groups of $X$ and $X_4$ are isomorphic, it is enough to show that $X_4$ is an 
Eilenberg-MacLane space.
Consider the Kummer morphism $\kappa:\PP^2\to \PP^2$ defined by $\kappa([x:y:z])=[x^2:y^2:z^2]$.
Note that $\kappa$ is a 4:1 ramified cover and its ramification locus is $R=\{xyz=0\}$. Since $R\subset \cC$
$\kappa$ defines an unramified cover on $X_4=\kappa^{-1}(X)$. Moreover, the preimage of $\cC$ by $\kappa$ is a product 
of 7 lines, three of which are the axis $xyz=0$ and four of them are the preimage of the conic $(x^2+y^2+z^2-2(xy+xz+yz))=0$.
In particular
$$
\array{rl}
\cC_2=\kappa^{-1}(\cC)=&\{F(x^2,y^2,z^2)=0\}\\
=&\{xyz(x+y+z)(x+y-z)(x-y+z)(x-y-z)=0\}.
\endarray
$$
Geometrically this corresponds to a Ceva arrangement --\,formed by the six lines of a generic pencil of conics\,--
with an extra line passing through two out of the three double points. In our equations, the pencil of conics
can be defined as $F_{[\alpha:\beta]}=\alpha ((x+z)^2-y^2)-\beta ((x-z)^2-y^2)$. Note that for $\alpha=\beta=1$
one obtains $F_{[1:1]}=4xz$. The rational map $\pi:\PP^2\dashrightarrow \PP^1$ defined by the pencil, where 
$\pi^{-1}[\alpha:\beta]=\{F_{[\alpha:\beta]}=0\}$, that is, $\pi([x:y:z])=[(x-z)^2-y^2:(x+z)^2-y^2]$ is not 
defined at the base points of the pencil. Since the curve $\cC_2$ contains these base points, one obtains that
$\pi|_{X_4}$ is well defined, where $X_4=\kappa^{-1}(X)=\PP^2\setminus \cC_2$. 

After our discussion above, recall that the special fibers of $\pi$ are the six lines 
$\{xz(x+y+z)(x+y-z)(x-y+z)(x-y-z)=0\}$. Finally, note that the line $y=0$ is a multisection since $\pi|_{y=0}$ 
is defined by $\pi([x:0:z])=[(x-z)^2:(x+z)^2]$ which is 2:1 and ramifies only at $[0:1]$ and $[1:0]$, therefore the map 
$$
\array{rcl}
\pi|_{X_4}: X_4 & \to & \PP^1\setminus \{[0:1],[1:0],[1:1]\} \\
{[x:y:z]} & \mapsto & [(x-z)^2-y^2:(x+z)^2-y^2]
\endarray
$$
is a well-defined locally trivial fibration whose generic fiber is the smooth conic of the pencil with six points 
removed (the four base points and the two points of intersection with the multisection $\{y=0\}$). Therefore 
$X_4$ is an Eilenberg-MacLane space.

\section{An example}
To end this paper we take a closer look into the triangle Artin group $\A_T$, $T=T(4,4,2)$ given by geometrical methods
coming from its quasi-projectivity property.

First we will show that $\A_T$ is not an extension of free groups. To do so we first study the surjections of $\A_T$ onto
a free group $\FF_r$ of rank $r$. Any surjection of groups $G_1\surj G_2$ induces an injection of characteristic varieties
$V_i(G_2)\hookrightarrow V_i(G_1)$ via the change of base $*\otimes_{\CC[G_2/G'_2]}\CC[G_1/G'_1]$ that turns an ideal in 
$\CC[G_2/G'_2]$ to an ideal in $\CC[G_1/G'_1]$ (see~\cite{Libgober-characteristic}). An Alexander matrix of $\A_T$ can be 
obtained immediately from Lemma~\ref{lemma-foxA} and~\eqref{eq-row} as
$$
M_{\A_T}=
\newcommand{\Bold}[1]{\mathbf{#1}}\left(\begin{array}{rrr}
-{\left(t_{0} t_{1} + 1\right)} {\left(t_{1} - 1\right)} & {\left(t_{0} t_{1} + 1\right)} {\left(t_{0} - 1\right)} & 0 \\
-{\left(t_{0} t_{2} + 1\right)} {\left(t_{2} - 1\right)} & 0 & {\left(t_{0} t_{2} + 1\right)} {\left(t_{0} - 1\right)} \\
0 & -t_{2} + 1 & t_{1} - 1
\end{array}\right)
$$
and thus its characteristic variety $V_1(\A_T)=\TT_1\cup \TT_2\cup \TT_3$ is the zero set of the Fitting 
ideal generated by the $2\times 2$-minors of $M_{\A_T}$, where 
$$
\array{c}
\TT_1=\{(-t^{-1},t,1)\mid t\in \CC^*\}\subset (\CC^*)^3,\\
\TT_2=\{(-t^{-1},1,t)\mid t\in \CC^*\}\subset (\CC^*)^3,\\
\TT_3=\{(-t^{-1},t,t)\mid t\in \CC^*\}\subset (\CC^*)^3\\
\endarray
$$
are three one-dimensional complex tori in $(\CC^*)^3$. Since the characteristic variety of the free group 
$\FF_r$ has dimension $r$, this implies that the only possible surjection $\A_T\surj \FF_r$ is restricted to $r=1$.

Note that any short exact sequence
$$
1\to \FF_s\to \A_T\to \ZZ\to 0
$$
splits and the action of $\ZZ$ on $\A_T$ is trivial in homology. Therefore $\A_T=\FF_s\rtimes \ZZ$ is called an 
\emph{IA-product of free groups} and by~\cite[Corollary 3.4]{Cohen-Suciu-Homology} the Poincaré polynomial 
$P_{\A_T}(t)$ of $\A_T$ should factor as a product of linear terms in~$\ZZ[t]$. However, since the complement 
$X=\PP^2\setminus \cC$ of the conic an three tangent lines shown in Figure~\ref{fig-ceva} is a $K(\A_T,1)$-space 
it is enough to calculate $P_X(t)$. One can easily check that $h_0(X)=1$ and $h_1(X)=3$. Moreover, using the 
additivity of the Euler characteristic
$$
\array{rcl}
\chi(X) & = &\chi(\PP^2)-\sum \chi(\cC_i)+\# \Sing(\cC)=
3-4\chi(\PP^1)+3=1\\
& = & h_0(X)-h_1(X)+h_2(X)=-2+h_2(X),
\endarray
$$
where $\cC_i$ are the irreducible components of $\cC$ and $\chi(\cC_i)=\chi(\PP^1)=2$ since they are all rational curves.
Therefore $h_2(X)=3$ and thus
$$
P_{\A_T}(t)=P_{X}(t)=3t^2+3t+1
$$
which is not a product of linear factors in~$\ZZ[t]$.

However, as shown in the proof of Theorem~\ref{thm-kpi1}, --\,see section~\ref{sec-proofkpi1}\,-- 
its 4-fold cover $X_4$ is the complement of a line arrangement of fibered type whose fundamental group 
$\pi_1(X_4)$ is a finite index normal subgroup of $\A_T$ which is an IA-free product of free 
groups~$\FF_3\rtimes\FF_3$.

\end{document}